\documentclass[11pt]{article}

\usepackage{color}
\usepackage{amsfonts}
\usepackage{amscd}
\usepackage{amssymb}
\usepackage{amsthm}
\usepackage{amsmath, xspace}
\usepackage[all]{xy}
\usepackage{graphics}
\usepackage{graphicx}
\usepackage{enumitem}
\usepackage{fancyhdr}
\usepackage{lscape}

\theoremstyle{plain}
\newtheorem{thm}{Theorem}[section]

\newtheorem{prop}[thm]{Proposition}

\theoremstyle{definition}

\newtheorem{remark}[thm]{Remark}
\theoremstyle{example}

\theoremstyle{remark}

\numberwithin{equation}{section}

\providecommand{\keywords}[1]{\textbf{\textit{Key words---}} #1}

\def\cD{\mathcal{D}}

\def\cF{\mathcal{F}}

\def\cL{\mathcal{L}}

\def\cO{\mathcal{O}}
\def\cP{\mathcal{P}}

\def\CC{\mathbb{C}}

\def\FF{\mathbb{F}}

\def\QQ{\mathbb{Q}}
\def\RR{\mathbb{R}}

\def\ZZ{\mathbb{Z}}

\newcommand\bbZ{\mathbb{Z}}

\def\fa{\mathfrak{a}}
\def\fb{\mathfrak{b}}

\def\fg{\mathfrak{g}}
\def\fh{\mathfrak{h}}

\def\fgl{\mathfrak{gl}}
\def\fsl{\mathfrak{sl}}

\def\Card{\mathrm{Card}}

\def\dim{\mathrm{dim}}

\def\Hom{\mathrm{Hom}}

\usepackage{etex} 
\usepackage{pictexwd}
\usepackage{tikz}
\usepgflibrary{shapes.geometric}
\usetikzlibrary{arrows}

\tikzstyle{V}=[draw, fill =black, circle, inner sep=0pt, minimum size=1.5pt]
\tikzstyle{wV}=[draw, fill =white, circle, inner sep=0pt, minimum size=4.5pt]
\tikzstyle{bV}=[draw, fill =black, circle, inner sep=0pt, minimum size=4.5pt]
\tikzstyle{over}=[draw=white,double=black,line width=2pt, double distance=.5pt]

\def\Over[#1,#2][#3,#4]{ 
	\draw[style=over]   (#1,#2) .. controls ++(0,#4*.5-#2*.5) and ++(0,-#4*.5+#2*.5) .. (#3,#4);}
\def\Under[#1,#2][#3,#4]{ 
	\draw  (#1,#2) .. controls ++(0,#4*.5-#2*.5) and ++(0,-#4*.5+#2*.5) .. (#3,#4);}
\def\Cross[#1,#2][#3,#4]{
	\Under[#3,#2][#1,#4]\Over[#1,#2][#3,#4]}

\def\Tops[#1][#2][#3]{
	\foreach\x in {#1}{
		\draw (\x+.1,#2) -- (\x+.1,#2+.15) (\x-.1,#2) -- (\x-.1,#2+.15) ;
		\draw (\x+.1,#2+.15) arc (0:360:1mm and .5mm);}
	\foreach \x in {1,...,#3} {\draw (\x,#2)  to (\x,#2+.05) node[V]{};}
	}
\def\Bottoms[#1][#2][#3]{
	\foreach\x in {#1}{
		\draw (\x+.1,#2) -- (\x+.1,#2-.1) (\x-.1,#2) -- (\x-.1,#2-.1) ;
		\draw (\x+.1,#2-.1) arc (0:-180:1mm and .5mm);}
	\foreach \x in {1,...,#3} {\draw (\x,#2)  to (\x,#2-.05) node[V]{};}
	}
\def\Caps[#1][#2,#3][#4]{
	\Tops[#1][#3][#4]
	\Bottoms[#1][#2][#4]
	}
\def\Pole[#1][#2,#3]{
	\shade[left color=white,right color=white] (#1+.1,#2) rectangle (#1-.1,#3);
	\draw[over] (#1+.1,#2) to (#1+.1,#3) (#1-.1,#2) to (#1-.1,#3) ;}
\def\Label[#1,#2][#3][#4]{
	\node[above] at (#3,#2+.1) {#4};
	\node[below] at (#3,#1-.1) {#4};		}

\def\mapright#1{\smash{\mathop
        {\longrightarrow}\limits^{#1}}}

\makeatletter
\renewcommand{\@makefnmark}{\mbox{\textsuperscript{}}}
\makeatother

\title{A Fock space model for decomposition numbers for\\
quantum groups at roots of unity}
\author{
Martina Lanini \ \ email:\ lanini@mat.uniroma2.it \\
Arun Ram\quad\ \ email:\ aram@unimelb.edu.au \\
Paul Sobaje\quad\qquad\ \ email:\ sobaje@uga.edu \\
\\
}
\date{}

\begin{document}

\maketitle

\begin{abstract}
\noindent
In this paper we construct an ``abstract Fock space" for general Lie types that serves as a generalisation of the infinite wedge $q$-Fock space familiar in type $A$.  Specifically, for each positive integer $\ell$, we define a $\ZZ[q,q^{-1}]$-module $\mathcal{F}_{\ell}$ with bar involution by specifying generators and ``straightening relations" adapted from those appearing in the Kashiwara-Miwa-Stern formulation of the $q$-Fock space.  By relating $\mathcal{F}_{\ell}$ to the corresponding affine Hecke algebra we show that the abstract Fock space has standard and canonical bases for which the transition matrix produces parabolic affine Kazhdan-Lusztig polynomials.  This property and the convenient combinatorial labeling of bases of $\mathcal{F}_{\ell}$ by dominant integral 
weights makes $\cF_\ell$ a useful combinatorial tool for determining decomposition numbers of 
Weyl modules for quantum groups at roots of unity.
\end{abstract}


\setcounter{section}{-1}

\section{Introduction}

The classical Fock space arises in the context of mathematical physics, where one would like to describe the behaviour of certain configurations with an unknown number of identical, non-interacting particles. It is  a (non-irreducible) representation of the affine Lie algebra $\widehat{\mathfrak{sl_n}}$.  The book \cite{MJD}, 
for example, is an inspiring and friendly tour of applications and connections between this representation, 
integrable systems, hierarchies of differential equations and infinite dimensional Grassmannians.

Combinatorial models have proven to be incredibly useful in studying the representations of various algebraic objects, such as affine Lie algebras, algebraic groups, Lie algebras,  quantum groups and symmetric groups.  Often the goal is to express simple modules in terms of ``standard'' modules (modules whose dimensions and formal characters are computable).

In a wonderful confluence of these two points of view, Lascoux-Leclerc-Thibon \cite{LLT} predicted a connection
between Hayashi's $q$-Fock space \cite{Ha} and decomposition numbers for representations of type A 
Iwahori-Hecke algebras at roots of unity. The LLT conjecture was proved in about 1995, see the Seminar Bourbaki survey of Geck \cite{Ge}.  The book of Kleshchev \cite{Kl} shows how successful these methods
have been in the study of the modular representation theory of symmetric groups.

This paper arose from an effort to produce an object analogous to the $q$-Fock space that will 
play the same role in other Lie types, in particular which will be related to the decomposition numbers 
for representations of cyclotomic BMW algebras in the same way that the type A case is related to representations 
of cyclotomic Hecke algebras.

In this paper, we provide a construction of an ``abstract'' Fock space $\cF_\ell$ in a general Lie type setting.
Our construction is given by simple combinatorial ``straightening relations'' which
generalize the Kashiwara-Miwa-Stern \cite{KMS} formulation of the $q$-Fock space from the
type A case.  Adapting the methods used by Leclerc-Thibon \cite{LT} for the type A case, we prove that our 
abstract Fock space picks up  the parabolic affine Kazhdan-Lusztig polynomials for the corresponding 
affine Hecke algebra of the affine Weyl group (thus generalizing type A results of Varagnolo-Vasserot
\cite{VV}).
By a combination of the results of Kashiwara-Tanisaki \cite{KT95} and
Kazhdan-Lusztig \cite{KL94} and Shan \cite{Sh}, 
these parabolic affine Kazhdan-Lusztig polynomials are graded decomposition numbers of
Weyl modules for the corresponding affine Lie algebra at negative level  
and for the quantum group at a root of unity.

A combinatorial study of the same parabolic affine Kazhdan-Lusztig polynomials was carried out also in \cite{GW},
 where the authors provided an efficient algorithm which generalizes the algorithm appearing to \cite{LLT} to arbitrary Lie type.
The focus of \cite{GW} was the combinatorial understanding of such polynomials rather than 
the construction of a tool that can play the same role for other Lie types that the infinite wedge
space takes in the type A case.

In Section 1 we give the simple construction of the general Lie type ``abstract Fock space'' $\cF_\ell$.  We
then explain exactly how this general construction relates to the classical type A setting, the framework
of Kashiwara-Miwa-Stern and the familiar formulations in terms of semi-infinite wedges, partitions and Maya diagrams.
In Section 2 we give an expository treatment of modules with bar involution, general bar-invariant KL-bases, 
and the construction of KL-polynomials for Hecke algebras, including the singular, parabolic and
parabolic-singular cases.  Although this material is well known (see, for example, \cite{Soe97}, \cite{Lu83}, \cite{Lu90b}, \cite{Du}) it is crucial for us to set
this up in a form suitable for connecting to the abstract Fock space so that we can eventually see the parabolic affine KL-polynomials
in the abstract Fock space $\cF_\ell$.  In Section 3 we review the results of Kashiwara-Tanisaki, Kazhdan-Lusztig and Shan
and concretely connect the decomposition numbers for Weyl modules of affine Lie algebras at negative level and 
quantum groups at roots of unity to the parabolic and parabolic-singular KL polynomials that have been treated in Section 2.
In Section 4, we prove that a certain module with bar involution which is constructed from the affine Hecke algebra
is isomorphic to the abstract Fock space $\cF_\ell$.  This is the key step for proving that the abstract Fock space
picks up the appropriate parabolic and parabolic-singular  affine KL-polynomials.  
Finally, at the end of section 4 we tie together
the results of Section 3 and 4 to conclude that the abstract Fock space, a combinatorial construct, computes the 
decomposition numbers of Weyl modules for quantum groups at roots of unity.

Our construction is an important first step in providing combinatorial tools for general Lie type that
are direct analogues of the tools that have been so useful in the Type A case.    There is much to be done.
In particular,  we hope that in the future someone will complete the following:
\begin{enumerate}
\item[(a)] Development of the combinatorics of $\cF_\ell$ in parallel to the way it is used in the
type A case (see, for example, Kleshchev's book \cite{Kl}) to provide a ``theory of crystals'' for other types
which applies to the representation theory of the cyclotomic BMW algebras in the same way that the 
classical crystal theory applies to the modular representation theory of cyclotomic Hecke algebras.
\item[(b)] Provide operators on $\cF_\ell$ analogous to the $U_q\widehat\fsl_\ell$ action on $\cF_\ell$
in the type A case.  Taking the point of view of \cite{RT} these operators are the (graded Grothendieck
group) images of translation functors for representations of the quantum group at a root of unity.  There
is significant evidence (see, for example, \cite{ES13}, \cite{BW}, \cite{BSWW} and \cite{FLLLW}) leading one to expect that in the type 
$B,C$ and $D$ cases these operators will
provide actions of  coideal quantum groups on $\cF_\ell$. 
\item[(c)] Elias-Williamson \cite{EW} introduced the diagrammatic Hecke category $\cD_{\mathrm{BS}}$ over a field, which in characteristic zero provides a generators and relations presentation of the Soergel bimodule category. It is expected \cite[Conjecture 5.1]{RW} that a regular block $\mathrm{Rep}_0(G(\overline{\FF_p}))$ is equipped with  an action of the category $\cD_{\mathrm{BS}}$ over $\overline{\FF_p}$. This conjecture can be viewed as a (categorical) extension of the project described in (b).
Indeed, our abstract Fock space $\cF_p$ is designed to be a decategorification of 
$\mathrm{Rep}(G(\overline{\FF_p}))$.
For the type A case, Riche-Williamson \cite{RW} have used the $U(\widehat\fgl_p)$-action on $\cF_p$
(in its infinite wedge space formulation) to prove their conjecture and hence to show that the $p$-canonical basis corresponds to the 
indecomposable tilting modules in $\mathrm{Rep}_0(G(\overline{\FF_p}))$.  It is possible that our abstract Fock space $\cF_p$ could be a useful
tool for generalizing the results of \cite{RW} to other Lie types in a uniform fashion (taking care also of singular blocks).
\end{enumerate}

It is a pleasure to thank all the institutions which have supported our work on this paper, 
including especially the University of Melbourne,  the Australian Research Council (grants DP1201001942 and DP130100674) and ICERM (Institute for Computational and Experimental Research in Mathematics). M.L. would like to thank the University of Edinburgh, which  supported her research during the final part of this project.

\section{The abstract Fock space}

\subsection{Fock space $\cF_\ell$}\label{subsectionFelldefin}

Let $W_0$ be a finite Weyl group, generated by simple reflections
$s_1, \ldots, s_n$, and acting on a lattice of weights $\fa_\ZZ^*$.  
For example, this situation arises when $T$ is a maximal torus of a reductive algebraic group $G$,
\begin{equation}
\fa_\ZZ^* = \Hom(T,\CC^\times)
\qquad\hbox{and}\qquad
W_0 = N(T)/T,
\label{wtsWeylgpdefn}
\end{equation}
where $N(T)$ is the normalizer of $T$ in $G$.  The simple reflections
in $W_0$ correspond to a choice of Borel subgroup $B$ of $G$ which contains $T$.
Let $R^+$ denote the positive roots.  Let $\alpha_1, \ldots, \alpha_n$ be the simple roots and let
$\alpha_1^\vee, \ldots, \alpha_n^\vee$ be the simple coroots.
The \emph{dot action} of $W_0$ on $\fa_\ZZ^*$ is given by
\begin{equation}
w\circ\lambda = w(\lambda+\rho)-\rho,
\qquad\hbox{where}\quad \rho = \hbox{$\frac12$}\sum_{\alpha\in R^+} \alpha
\label{dotaction}
\end{equation}
is the half sum of the positive roots for $G$ (with respect to $B$).

Fix $\ell\in \ZZ_{>0}$.  The \emph{Fock space} $\mathcal{F}_\ell$ is
the $\ZZ[t^{\frac12},t^{-\frac12}]$-module generated by $\{ \vert \lambda\rangle \ |\ \lambda\in \fa_{\ZZ}^{*}\}$ with relations
\begin{equation}
\vert s_i\circ\lambda\rangle=\begin{cases}
-\vert \lambda\rangle, &\hbox{if $\langle\lambda+\rho,\alpha_i^\vee\rangle \in \ell\ZZ_{\ge 0}$,} \\
-t^{\frac12}\vert \lambda\rangle, &\hbox{if $0<\langle\lambda+\rho,\alpha_i^\vee\rangle<\ell $,} \\
-t^{\frac12}\vert s_i\circ\lambda^{(1)} \rangle -\vert \lambda^{(1)} \rangle -t^{\frac12}\vert \lambda\rangle, 
&\hbox{if $ \langle\lambda+\rho,\alpha_i^\vee\rangle > \ell$ and $\langle\lambda+\rho,\alpha_i^\vee\rangle\not\in \ell\ZZ$,}
\end{cases}
\label{Fstraightening}
\end{equation}
where 
$\lambda^{(1)} = \lambda - j\alpha_i$
if $\langle\lambda+\rho, \alpha_i^\vee\rangle = k\ell + j$ with $k\in \ZZ_{>0}$ and $j\in \{1, \ldots, \ell-1\}$.

The following picture illustrates the terms in \eqref{Fstraightening}.  This is the case $G=SL_2$ with $\ell=5$, 
$\langle\omega_1,\alpha_1^\vee\rangle=1$ and $\alpha_1=2\omega_1$ and, in the picture,
$\lambda$ corresponds to the third case of \eqref{Fstraightening}, $\mu$ to the first case and $\nu$ to the second case.
$$\setlength{\unitlength}{0.5cm}
\begin{picture}(23,3)
\put(-5,2){\line(1,0){32}}

\put(-4,2){\circle*{0.3}}
\put(-3,2){\circle*{0.3}}
\put(-2,2){\circle*{0.3}}
\put(-1,2){\circle*{0.3}}
\put(0,2){\circle*{0.3}}
\put(1,2){\circle*{0.3}}
\put(2,2){\circle*{0.3}}
\put(3,2){\circle*{0.3}}
\put(4,2){\circle*{0.3}}
\put(5,2){\circle*{0.3}}
\put(6,2){\circle*{0.3}}
\put(7,2){\circle*{0.3}}
\put(8,2){\circle*{0.3}}
\put(9,2){\circle*{0.3}}
\put(10,2){\circle*{0.3}}
\put(11,2){\circle*{0.3}}
\put(12,2){\circle*{0.3}}
\put(13,2){\circle*{0.3}}
\put(14,2){\circle*{0.3}}
\put(15,2){\circle*{0.3}}
\put(16,2){\circle*{0.3}}
\put(17,2){\circle*{0.3}}
\put(18,2){\circle*{0.3}}
\put(19,2){\circle*{0.3}}
\put(20,2){\circle*{0.3}}
\put(21,2){\circle*{0.3}}
\put(22,2){\circle*{0.3}}
\put(23,2){\circle*{0.3}}
\put(24,2){\circle*{0.3}}
\put(25,2){\circle*{0.3}}
\put(26,2){\circle*{0.3}}

\put(-4.8,0.8){$\scriptstyle{-14}$}
\put(-3.8,0.8){$\scriptstyle{-13}$}
\put(-2.8,0.8){$\scriptstyle{-12}$}
\put(-1.8,0.8){$\scriptstyle{-11}$}
\put(-0.8,0.8){$\scriptstyle{-10}$}
\put(0.4,0.8){$\scriptstyle{-9}$}
\put(1.4,0.8){$\scriptstyle{-8}$}
\put(2.4,0.8){$\scriptstyle{-7}$}
\put(3.4,0.8){$\scriptstyle{-6}$}
\put(4.4,0.8){$\scriptstyle{-5}$}
\put(5.4,0.8){$\scriptstyle{-4}$}
\put(6.4,0.8){$\scriptstyle{-3}$}
\put(7.4,0.8){$\scriptstyle{-2}$}
\put(8.4,0.8){$\scriptstyle{-\rho}$}
\put(9.85,0.8){$\scriptstyle{0}$}
\put(10.85,0.8){$\scriptstyle{\omega_1}$}
\put(11.85,0.8){$\scriptstyle{2}$}
\put(12.85,0.8){$\scriptstyle{3}$}
\put(13.85,0.8){$\scriptstyle{4}$}
\put(14.85,0.8){$\scriptstyle{5}$}
\put(15.85,0.8){$\scriptstyle{6}$}
\put(16.85,0.8){$\scriptstyle{7}$}
\put(17.85,0.8){$\scriptstyle{8}$}
\put(18.85,0.8){$\scriptstyle{9}$}
\put(19.75,0.8){$\scriptstyle{10}$}
\put(20.75,0.8){$\scriptstyle{11}$}
\put(21.75,0.8){$\scriptstyle{12}$}
\put(22.75,0.8){$\scriptstyle{13}$}
\put(23.75,0.8){$\scriptstyle{14}$}
\put(24.75,0.8){$\scriptstyle{15}$}
\put(25.75,0.8){$\scriptstyle{16}$}

\put(-1,1.2){\line(0,1){1.5}}
\put(4,1.2){\line(0,1){1.5}}
\put(9,1.2){\line(0,1){1.5}}
\put(14,1.2){\line(0,1){1.5}}
\put(19,1.2){\line(0,1){1.5}}
\put(24,1.2){\line(0,1){1.5}}


\put(19.75,2.3){$\scriptstyle{\lambda}$}
\put(17.75,2.3){$\scriptstyle{\lambda^{(1)}}$}
\put(-0.8,2.3){$\scriptstyle{s_1\circ\lambda^{(1)} }$}
\put(-2.8,2.3){$\scriptstyle{s_1\circ\lambda}$}

\put(18.75,3){$\scriptstyle{\mu}$}
\put(-1.8,3){$\scriptstyle{s_1\circ\mu}$}

\put(11.75,2.3){$\scriptstyle{\nu}$}
\put(5.4,2.3){$\scriptstyle{s_1\circ\nu}$}


\end{picture}
$$
Define a $\ZZ$-linear involution 
$\overline{\phantom{T}}\colon \mathcal{F}_{\ell}\to \mathcal{F}_{\ell}$ by
\begin{equation}
\overline{t^{\frac12}} = t^{-\frac12}
\qquad\hbox{and}\qquad
\overline{\vert \lambda \rangle} = (-1)^{\ell(w_0)}(t^{-\frac12})^{\ell(w_{0})-N_\lambda}\, \vert w_0\circ \lambda \rangle.
\label{Fellbar}
\end{equation}
where $w_0$ is the longest element of $W_0$, $\ell(w_0) = \Card(R^+)$ is the length of $w_0$, and 
$N_\lambda = \Card\{\alpha\in R^+\ |\ \langle \lambda+\rho, \alpha^\vee\rangle \in \ell\ZZ\}.$

\subsection{$\cF_\ell$ is a KL-module}

The \emph{dominant integral weights}  with the \emph{dominance partial order} $\le$ are the elements of 
\begin{equation}
\begin{array}{c}
(\fa_\ZZ^*)^+ 
= \{\lambda\in \fa_\ZZ^*\ |\ \hbox{$\langle \lambda+\rho, \alpha_i^\vee\rangle> 0$ for $i=1, 2,\ldots, n$} \}
\\
\\
\hbox{with}\qquad
\mu\le \lambda\qquad\hbox{if}\quad
\mu \in \lambda - \sum_{\alpha\in R^+} \ZZ_{\ge 0}\alpha.
\end{array}
\label{domintwtsdefn}
\end{equation}
In combination, Theorem \ref{Fock_is_KL} and Proposition \ref{Prop_KLBases} below 
give that $\cF_\ell$ has bases
\begin{equation}
\{ \vert\lambda\rangle\ |\ \lambda\in (\fa_\ZZ^*)^+\}
\qquad\hbox{and}\qquad
\{ C_\lambda\ |\ \lambda\in (\fa_\ZZ^*)^+\}
\label{Fellbases}
\end{equation}
where $C_\lambda$ are determined by 
\begin{equation}
\overline{C_\lambda} = C_\lambda
\qquad\hbox{and}\qquad
C_\lambda = \vert \lambda\rangle + \sum_{\mu\ne \lambda} p_{\mu\lambda} \vert\mu\rangle,
\qquad\hbox{with $p_{\mu\lambda}\in t^{\frac12}\ZZ[t^{\frac12}]$.}
\end{equation}

\begin{thm}\label{Fock_is_KL}  Let $\cF_\ell$ be defined as \eqref{Fstraightening} and let 
$\cL = \{ \vert\lambda\rangle\ |\ \lambda\in (\fa_\ZZ^*)^+\}$.  Then, with the definition of KL-module as 
in Section 2, $\cL$ is a basis of $\cF_\ell$ and
$$((\fa_\ZZ^*)^+, \cF_\ell, \cL, 
\overline{\phantom{T}}\colon \mathcal{F}_{\ell}\to \mathcal{F}_{\ell})
\quad\hbox{is a KL-module.}$$
\end{thm}
\begin{proof} (Sketch)
If $\lambda \in (\fa_{\ZZ}^{*})^{+}$ then there are only finitely many $\mu \le \lambda$ with the property that $\mu$ 
is also dominant (see \cite[Cor. 1.4]{St}).

Let $i\in \{1, \ldots, n\}$ and let $\lambda\in \fa_\ZZ^*$
be such that
$0<\langle \lambda+\rho, \alpha_i^\vee\rangle$.  Write
$$\langle \lambda+\rho, \alpha_i^\vee\rangle  = k\ell + j,
\qquad\hbox{with $k\in \ZZ$ and $j\in \{0,1,\ldots, \ell-1\}$.}
$$
When $j\ne 0$ define
$$\lambda^{(1)} = \lambda - j\alpha_i
\quad\hbox{and}\quad
\lambda^{(j+1)} = (\lambda^{(j)})^{(1)}.$$
Then induction on $k$ using the third case in \eqref{Fstraightening} gives
\begin{align}
\vert s_i\circ \lambda \rangle 
&=(-t^{\frac12}) \vert \lambda \rangle + (-t^{\frac12}) t^{-\frac12} \vert \lambda^{(1)}\rangle + (-t^{\frac12}) \vert s_i\circ \lambda^{(1)}\rangle \nonumber \\
& = (-t^{\frac12})\vert \lambda \rangle + (-t^{\frac12})t^{-\frac12} \vert \lambda^{(1)} \rangle \nonumber \\
&\qquad +(-t^{\frac12})(-t^{\frac12})\left(
\begin{array}{l}
\vert \lambda^{(1)} \rangle
- (t^{\frac12}-t^{-\frac12})\vert \lambda^{(2)}\rangle
- (t^{\frac12}-t^{-\frac12})(-t^{\frac12})\vert \lambda^{(3)}\rangle \\
\qquad - \cdots - (t^{\frac12}-t^{-\frac12}) (-t^{\frac12})^{k-2} \vert\lambda^{(k)}\rangle
\end{array}
\right) \nonumber \\
&=(-t^{\frac12})\left(
\begin{array}{l}
\vert \lambda \rangle
- (t^{\frac12}-t^{-\frac12})\vert \lambda^{(1)}\rangle
- (t^{\frac12}-t^{-\frac12})(-t^{\frac12})\vert \lambda^{(2)}\rangle \\
\qquad - \cdots - (t^{\frac12}-t^{-\frac12})(-t^{\frac12})^{k-1}\vert\lambda^{(k)}\rangle
\end{array}\right).
\label{rank1straightened}
\end{align}
More generally, for $\lambda \in \fa_\ZZ^*$ such that $\langle \lambda+\rho, \alpha_i^\vee\rangle \ne 0$ for $i\in \{1,\ldots, n\}$
let $\lambda^+$ be the dominant representative of $W_0\circ \lambda$ and let 
\begin{equation}
\begin{array}{rl}
R(\lambda) &= \{ \alpha\in R^+\ |\ \langle \lambda+\rho, \alpha^\vee\rangle \in \ZZ_{<0}\}, 
\\
R_\ell(\lambda) &= \{ \alpha\in R^+\ |\ \langle \lambda+\rho, \alpha^\vee\rangle \in \ell\ZZ_{<0}\}.
\end{array}
\end{equation}
Then iterating \eqref{rank1straightened} produces $c_\mu\in (t^{-\frac12}-t^{\frac12})\ZZ[t^{\frac12}]$ so that 
\begin{equation}
\vert \lambda\rangle
= (-1)^{\Card(R(\lambda))}(t^{\frac12})^{\Card(R(\lambda))-\Card(R_\ell(\lambda))}\left(\vert\lambda^+\rangle
+ \sum_{\mu^+\in (\fa_\ZZ^*)^+\atop  \mu^+ \le \lambda^+} c_{\mu} \vert \mu^+ \rangle\right).
\label{lemma_straightening3}
\end{equation}

\noindent
With \eqref{lemma_straightening3} in hand all steps in a direct proof of Theorem \ref{Fock_is_KL}
are straightforward except proving that $\{ \vert\lambda^+\rangle\ |\ \lambda^+\in (\fa_\ZZ^*)^+\}$ is a basis of $\cF_\ell$
(the linear independence is the issue).  To prove this directly the unpleasant step is to show that
if $\lambda^+\in (\fa_\ZZ^*)^+$ and $w\in W_0$ then $\vert w\circ\lambda^+\rangle$ defined
by $\vert w\circ \lambda^+\rangle = \vert s_{i_1}\circ(s_{i_2}\circ \cdots \circ (s_{i_k}\circ \lambda^+))\rangle$ for a reduced decomposition
$w=s_{i_1}s_{i_2}\cdots s_{i_k}$ will produce a well defined element of $\cF_\ell$ (independent of the choice of reduced 
decomposition).  Alternatively, it is possible to use a Gr\"obner basis argument using the ordering $\preceq$ on $\fa_\ZZ^*$ given by
$$
\begin{array}{cl}
\mu\prec \lambda\quad &\hbox{if $\mu^+< \lambda^+$ in dominance order and} \\
u\circ \lambda^+ \prec v\circ\lambda^+
&\hbox{if $u<v$ in Bruhat order,}
\end{array}
$$
where $\mu^+$ denotes the dominant representative of $W_0\circ \mu$.
However, we will not complete this sketch here as Theorem \ref{Fock_is_KL} is a consequence of
the realization of $\cF_\ell$ provided by Theorem \ref{abstracttoHecke}.
\end{proof}

\subsection{$\cF_\ell$ as a semi-infinite wedge space for the case $G=GL_\infty$}

Fix $\ell\in \ZZ_{>0}$.  
The semi-infinite wedge space considered by Kashiwara-Miwa-Stern \cite[(43)-(45)]{KMS} is
\begin{equation}
\cF_\ell = \Lambda^{\frac{\infty}{2}}V 
= \hbox{$\CC$-span}\left\{ v_{a_1}\wedge v_{a_2}\wedge\cdots\ \Bigg|\ \begin{matrix}
\hbox{$a_j\in \ZZ$ and, for all but} \\
\hbox{a finite number of $j$, $a_j = -j+1$}
\end{matrix}\right\},
\label{semiinfwedgespace}
\end{equation}
where $v_a$, $a\in \ZZ$ are symbols, and if $a<b$ then
$$v_b\wedge v_a = \begin{cases}
-(v_a\wedge v_b), &\hbox{if $a-b\in \ell\ZZ_{\ge 0}$,} \\
-t^{\frac12}(v_a\wedge v_b), &\hbox{if $0< a-b<\ell$,} \\
-t^{\frac12}(v_{b+j}\wedge v_{a-j})-(v_{a-j}\wedge v_{b+j})-t^{\frac12}(v_a\wedge v_b), &
\begin{matrix}\hbox{if $a-b=k\ell+j$ with $k\in \ZZ$} \\
\hbox{and $j\in \{0,1,\ldots, \ell-1\}$.}
\end{matrix}
\end{cases}
$$
From the point of view of \eqref{wtsWeylgpdefn} and \eqref{Fstraightening}, this is the case $G= GL_\infty(\CC)$ with
$\fa_\ZZ^* = \hbox{$\ZZ$-span}\{ \varepsilon_1, \varepsilon_2, \ldots\}$
and
$W_0$ the infinite symmetric group generated by $s_1, s_2, s_3, \ldots$, where
$s_i$ is the simple transposition that switches $\varepsilon_i$ and $\varepsilon_{i+1}$.
This framework illustrates that the straightening laws of \eqref{Fstraightening} are generalizations
of those that appear in \cite[(43-45)]{KMS} and \cite[Prop.\ 5.11]{LT}.

In the semi-infinite wedge space setting of \eqref{semiinfwedgespace} 
the bar involution appears in \cite[\S 3.6]{Le}, and \cite[Prop.\ 5.9 and (85)]{LT}. 
Kashiwara-Miwa-Stern \cite{KMS} already have the affine Hecke algebra playing a significant 
role in their story; in retrospect, this is not unrelated to the role that the affine Hecke algebra 
takes for us in Theorem \ref{abstracttoHecke}.   Leclerc-Thibon \cite{LT} also
have the affine Hecke algebra playing an important role, essentially the same as in this paper.

The correspondence between partitions, semi-infinite wedges and Maya diagrams appears in \cite[\S4.3 and Fig.\ 9.3]{MJD} (see also
\cite[\S2.2.1]{Le} and \cite[Fig.\ 1]{Tin}).
Following  \cite[\S2.2.1]{Le}, the partition 
$$
\lambda =(\lambda_1\ge \lambda_2\ge \cdots\ge \lambda_s>0)=(\lambda_1, \lambda_2, \ldots, \lambda_s, 0,0,\ldots)
\qquad\hbox{corresponds to}$$
$$
\hbox{the semi-infinite wedge} \qquad
|\lambda\rangle =v_{\lambda_1-1+1}\wedge v_{\lambda_2-2+1}\wedge \cdots\,.
$$ 
The $\rho$-shift which appears in \eqref{Fstraightening} also appears here since
$\rho$ can be taken to be
$$\rho=(0, 1, 2, 3, \ldots)
\qquad\hbox{for the case of $G=GL_\infty(\CC)$.}
$$
In the picture below, when following the bold boundary of the partition 
$\lambda=(4,4,3,3,2,2,1,1,1)$ the positive slope edges 
correspond to black dots in the Maya diagram and the black dots in the Maya diagram correspond to the indices in the corresponding wedge 
$\vert \lambda\rangle = v_{i_1}\wedge v_{i_2}\wedge \cdots$.

$$\includegraphics[height=8.6cm]{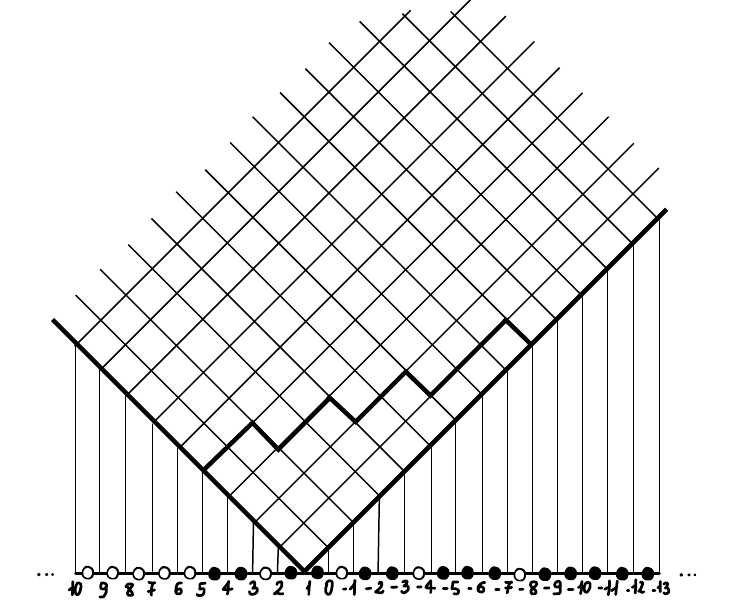}$$

$$\lambda = (4,4,3,3,2,2,1,1,1)
\quad\hbox{with}\quad
\vert\lambda\rangle = v_4\wedge v_3\wedge v_1\wedge v_0\wedge v_{-2}\wedge v_{-3}
\wedge v_{-5}\wedge v_{-6}\wedge v_{-7}\wedge v_{-9}\wedge v_{-10}\wedge \cdots.
$$

\section{KL-modules and bases}

The \emph{bar involution} on the ring $\ZZ[t^{\frac12}, t^{-\frac12}]$ of Laurent polynomials in $t^{\frac12}$ is 
the ring isomorphism 
\begin{equation}
\overline{\phantom{T}}: \ZZ[t^{\frac12}, t^{-\frac12}] \to \ZZ[t^{\frac12}, t^{-\frac12}]
\qquad\hbox{given by}\qquad
\overline{t^{\frac12}} = t^{-\frac12}.
\label{babybar}
\end{equation}
A \emph{KL-module} over $\ZZ[t^{\frac12}, t^{-\frac12}]$ is a tuple 
$(\Lambda,M,\{T_w\}_{w\in \Lambda}, \overline{\phantom{T}}\colon M\to M)$ where
\begin{enumerate}
\item[(a)] $\Lambda$ is a partially ordered set such that if $w \in \Lambda$ then
$\{ v\in \Lambda\ |\  v \le w \}$ is finite,
\item[(b)] $M$ is a free $\ZZ[t^{\frac12}, t^{-\frac12}]$-module with basis $\{T_w\ |\  w \in \Lambda\}$,
\item[(c)] $\overline{\phantom{T}}\colon M\to M$ is a $\ZZ$-module homomorphism such that
if $m\in M$, $a\in \ZZ[t^{\frac12}, t^{-\frac12}]$ and $w\in \Lambda$ then
\begin{equation}
\overline{a\cdot m} = \overline{a}\cdot \overline{m}, \qquad
\overline{\overline{m}} = m,
\qquad\hbox{and}\qquad
\overline{T_w} = T_w + \sum_{v<w} a_{vw} T_v,
\label{atriang}
\end{equation}
where $\overline{a}$ is given by \eqref{babybar} 
and the coefficients $a_{v,w}$ in the expansion of $\overline{T_w}$ are elements of $\ZZ[t^{\frac12}, t^{-\frac12}]$.
\end{enumerate}

\begin{prop}\label{Prop_KLBases}
Let $(\Lambda,M,\{T_w\}, \overline{\, \cdot \,})$ be a KL-module over $\ZZ[t^{\frac12}, t^{-\frac12}]$.  There is a 
unique basis $\{C_w\ |\ w\in \Lambda\}$ of $M$ characterized by
\begin{equation}
\overline{C_w} = C_w
\quad\hbox{and}\quad
C_w = T_w + \sum_{v<w} p_{vw} T_v ,\quad
\hbox{with $p_{vw}\in t^{\frac12}\ZZ[t^{\frac12}]$ for $v<w$.}
\label{Cmprops} 
\end{equation}
Let  $d_{vw}$ be the coefficients in the expansion
\begin{equation}\label{PQdefn}
T_w = C_w + \sum_{v<w} d_{vw} C_v,\quad\hbox{with $d_{vw}\in t^{\frac12}\ZZ[t^{\frac12}]$ for $v<w$.}
\end{equation}
The polynomials $p_{uw}$ and $d_{uw}=0$ 
are specified, inductively, by the equations $p_{uw}=d_{uw}=0$ unless $u\le w$, $p_{ww}=d_{ww}=1$,
\begin{equation}\label{PQrecursion}
p_{uw}-\overline{p_{uw}}
= \sum_{u<z\leqslant w} a_{uz}\overline{p_{zw}}
\qquad\hbox{and}\qquad
d_{uw}-\overline{d_{uw}}=-\sum_{u\leq z<w} d_{uz}a_{z,w}.
\end{equation}
\end{prop}
\begin{proof}
The matrices $A = (a_{vw})$,  $P=(p_{vw})$ and $D=(d_{vw})$ defined by
\eqref{atriang} and \eqref{PQdefn} are all upper triangular with 1's on the diagonal.
Then
\begin{equation}
A\overline{A}=1,\quad P=A\overline{P}, \quad \overline{D}=DA\quad \text{and}\quad DP=1=PD,
\label{KLmechanics}
\end{equation}
since
\begin{align*}
T_w &= \overline{\overline{T_w}} 
= \sum_v \overline{a_{vw}T_v}
=\sum_{u,v} a_{uv}\overline{a_{vw}}T_u, \\
\sum_u p_{uw}T_v 
&= C_w
=\overline{C_w}
= \sum_v \overline{p_{vw}T_v}
=\sum_{u,v} \overline{p_{vw}}a_{uv}T_u, \quad\hbox{and} \\
C_w+\sum_{v<w}\overline{d_{vw}}C_v
&=\overline{T_w}
=\sum_{u\leq w}a_{u,w} T_u
=\sum_{v\leq u \leq w}a_{u,w} d_{vu} C_v.
\end{align*}
Letting $\displaystyle{
f = p_{uw}-\overline{p_{uw}}
= \sum_{k\in \ZZ} f_k (t^{\frac12})^k,
}$
\begin{equation}\label{eqn_calculatePuw}
f = p_{uw}-\overline{p_{uw}}
=(P-\overline{P})_{uw} = 
((A-1)\overline{P})_{uw} = (A\overline P-\overline{P})_{uw}
= \sum_{u<z\leqslant w} a_{uz}\overline{p_{zw}},
\end{equation}
and the identity
$$
\overline{f} 
= \overline{(p_{uw}-\overline{p_{uw}})}
= \overline{p_{uw}} -p_{uw} = - f
\qquad\hbox{implies}\qquad
f_k = -f_{-k}, \ \ \hbox{for $k\in \ZZ$.}
$$
Thus $p_{uw} =\displaystyle{\sum_{k\in \ZZ_{<0}} f_k (t^{\frac12})^k }$.  The derivation of the 
formula for the entries of $D$ is similar using $D-\overline{D}=D-DA$ and $a_{ww}=1$.
\end{proof}

 \subsection{KL modules associated to Hecke algebras of Coxeter groups}

Let $W$ be a Coxeter group generated by $s_0, s_1, \ldots, s_n$ so that 
\begin{equation}
s_i^2=1,
\qquad\hbox{and}\qquad
(s_is_j)^{m_{ij}}=1,
\quad\hbox{for $i\ne j$}
\end{equation}
($m_{ij}$ is allowed to be $\infty$, in which case, the expression $(s_is_j)^{m_{ij}}=1$ should be interpreted as
``$s_is_j$ has infinite order'').
Let $w\in W$. A \emph{reduced word for} $w$ is a sequence $s_{i_1}\cdots s_{i_r}$ of generators
with $w=s_{i_1}\cdots s_{i_r}$ and $r$ minimal. The \emph{length} of $w$ is 
$\ell(w)=r$ if $s_{i_1}\ldots s_{i_r}$ is a reduced word  for $w$.   The \emph{Bruhat order} $\leq$ on $W$ 
is given by $v\leq w$ if there is a reduced word $s_{j_1}\ldots s_{j_m}$ for $v$ which is a subword of a reduced word $s_{i_1}\ldots s_{i_r}$ for $w$.

The \emph{Hecke algebra of $W$} is the $\bbZ[t^{\frac12}, t^{-\frac12}]$-algebra $H$
with generators $T_0, T_1, \ldots, T_n$  and relations
\begin{equation}
T_{i}^2=(t^{\frac{1}{2}}-t^{-\frac{1}{2}})T_i+1
\qquad\hbox{and}\qquad
\underbrace{T_{i}T_jT_i\cdots }_{m_{ij}\ \mathrm{factors}}
=\underbrace{T_{j}T_iT_j\cdots}_{m_{ij}\ \mathrm{factors}}.
\label{eqnT2}
\end{equation}
For $w\in W$ define $T_w=T_{s_{i_1}}\ldots T_{s_{i_r}}$ for a reduced word 
$w=s_{i_1}\cdots s_{i_r}$. By \cite[Ch. 4, \S 2, Ex. 23)]{Bou},
$T_w$ does not depend on the choice of reduced word for $w$ and 
\begin{equation}
\{ T_w\ |\ w\in W\}
\quad\hbox{is a $\bbZ[t^{\frac12}, t^{-\frac12}]$-basis of $H$.}
\label{Hstandardbasis}
\end{equation}
Define a $\ZZ$-algebra automorphism $\overline{\phantom{T}}\colon H\to H$ by 
\begin{equation}
\overline{t^{\frac12}} = t^{-\frac12}
\qquad\hbox{and}\qquad
\overline{T_w} = T_{w^{-1}}^{-1} \quad\hbox{for $w\in W$.}
\label{Hbarinvolution}
\end{equation}
By the first relation in (\ref{eqnT2}), 
$T_i^{-1}=T_{i_1}-(t^{\frac{1}{2}}-t^{-\frac{1}{2}})$, so that
if $w=s_{i_1}\cdots s_{i_r}$ is a reduced word for $w\in W$ then 
\begin{align*}
\overline{T_{w}}
&=\overline{T_{i_1}\cdots T_{i_r}}
=T_{i_1}^{-1}\cdots T_{i_r}^{-1}
=\left(T_{i_1}-(t^{\frac{1}{2}}-t^{-\frac{1}{2}}) \right)\cdots \left(T_{i_r}-(t^{\frac{1}{2}}-t^{-\frac{1}{2}}) \right)\\
&=T_{w}+\sum_{v<w} a_{vw} T_v, 
\qquad\hbox{with $a_{vw}\in (t^{\frac12}-t^{-\frac12})\ZZ[t^{\frac12}-t^{-\frac12}]$.}
\end{align*}
With standard basis as in \eqref{Hstandardbasis} indexed by the poset $W$ and with bar involution as in \eqref{Hbarinvolution},
$$\hbox{$H$ is a KL-module over
$\bbZ[t^{\frac12},t^{-\frac12}]$}
$$
and, from Proposition  \ref{Prop_KLBases}, there is
a unique basis  $\{C_w\ |\ w\in W\}$ determined by
\begin{equation}
\overline{C_x} = C_x
\qquad\hbox{and}\qquad
C_x = \sum_{y\le x\atop y\in W} (-1)^{\ell(x)-\ell(y)}P_{y,x}(t^{\frac12}) T_y,
\label{LusztigC}
\end{equation}
with $P_{y,x}(t^\frac12)\in t^{\frac{1}{2}}\ZZ[t^{\frac{1}{2}}]$ for $y<x$.
The polynomials $P_{y,x}$ are the \emph{Kazhdan-Lusztig polynomials} for $H$.

\subsection{Singular and parabolic KL polynomials}\label{HeckeKLsection}

\subsubsection{The projectors}

Let $J, \gamma\subseteq \{0,1,\ldots n\}$ and let $W_\nu$ and 
$W_\gamma$ be the subgroups of $W$ generated by the corresponding 
simple reflections,
\begin{equation}
W_\nu = \langle s_j\ |\ j\in J\rangle 
\qquad\hbox{and}\qquad
W_\gamma = \langle s_k\ |\ k\not\in \gamma\rangle,
\qquad\hbox{respectively.}
\label{stabsandsubsets}
\end{equation}
Assume that $W_\nu$ and $W_\gamma$ are both finite.
Let $w_\nu$ be the longest element of $W_\nu$ and let $w_\gamma$ be the longest element of $W_\gamma$ and let
\begin{equation}
W_\nu(t) = \sum_{z\in W_\nu} t^{\ell(z)}
\qquad\hbox{and}\qquad W_\gamma(t) = \sum_{z\in W_\gamma} t^{\ell(z)}.
\label{Poincarepoly}
\end{equation}
Then
\begin{align}
\mathbf{1}_\nu &= \sum_{z\in W_\nu} (t^{-\frac12})^{\ell(w_\nu)-\ell(z)} T_z 
= (t^{-\frac12})^{\ell(w_\nu)}\sum_{z\in W_\nu} (t^{\frac12})^{\ell(z)} T_z,
\qquad\hbox{and} \nonumber \\
\varepsilon_\gamma &= \sum_{z\in W_\gamma} (-t^{\frac12})^{\ell(w_\gamma)-\ell(z)} T_z
= (-t^{\frac12})^{\ell(w_\gamma)}\sum_{z\in W_\gamma} (-t^{-\frac12})^{\ell(z)} T_z,
\label{projectors}
\end{align}
satisfy
$$\begin{array}{lclcl}
\overline{\mathbf{1}_\nu} = \mathbf{1}_\nu,
&\quad 
&T_{s_j}\mathbf{1}_\nu = t^{\frac12}\mathbf{1}_\nu \ \hbox{for $j\in J$,}
&\quad\hbox{and}\quad 
&\mathbf{1}_\nu^2 = (t^{-\frac12})^{\ell(w_\nu)} W_\nu(t) \mathbf{1}_\nu, 
\\
\overline{\varepsilon_\gamma} = \varepsilon_\gamma,
&\qquad 
&\varepsilon_\gamma T_{s_k} = -t^{-\frac12}\varepsilon_\gamma
\ \hbox{for $k\not\in \gamma$,}
&\hbox{and} 
&\varepsilon_\gamma^2 = (-t^{-\frac12})^{\ell(w_\gamma)}W_\gamma(t) \varepsilon_\gamma,
\end{array}
$$
and
$$\mathbf{1}_\nu = T_{w_\nu} + \sum_{x<w_\nu} h^-_{x,w_\nu}T_x
\qquad\hbox{and}\qquad
\varepsilon_\gamma = T_{w_\gamma} + \sum_{x<w_\gamma} h_{x,w_\gamma}T_x,
$$
with coefficients $h_{x,w_\nu}^-\in t^{-\frac12}\ZZ[t^{-\frac12}]$ and 
$h_{x,w_\gamma}\in t^{\frac12}\ZZ[t^{\frac12}]$.

\subsubsection{Singular block KL polynomials}

As in \eqref{stabsandsubsets}, let $W_\nu = \langle s_j\ |\ j\in J\rangle$ and let
$W^\nu$ be the set of minimal length coset representatives of the 
cosets in $W/W_\nu$.  The $\ZZ[t^{\frac12}, t^{-\frac12}]$-module
\begin{equation}
\hbox{$H\mathbf{1}_\nu$ has basis} \quad \{T_u\mathbf{1}_\nu \ |\ u\in W^\nu\}
\qquad\hbox{and}\qquad
\overline{\phantom{T}}\colon H\mathbf{1}_\nu \to H\mathbf{1}_\nu,
\label{singKLmodbasis}
\end{equation}
since
$\overline{\mathbf{1}_\nu} = \mathbf{1}_\nu$.
The Bruhat order $W^\nu$ is the restriction of the Bruhat order on $W$ to $W^\nu$ and,
with these structures, $H\mathbf{1}_\nu$ is a KL-module.

If $\varphi\colon H \to H\mathbf{1}_\nu$ is the surjective KL-module homorphism defined 
by right multiplication by $\mathbf{1}_\nu$ then
\begin{equation}
\hbox{$H\mathbf{1}_\nu$ has KL-basis} \quad \{ C_u\mathbf{1}_\nu\ |\ u\in W^\nu\},
\label{singularKLbasis}
\end{equation}
where $\{C_w\ |\ w\in W\}$ is the KL-basis of $H$.  
With notation as in \eqref{LusztigC},
\begin{equation}
C_x\mathbf{1}_\nu
= \sum_{y\le x\atop y\in W} (-1)^{\ell(x)-\ell(y)} P_{y,x}(t^\frac12) T_y\mathbf{1}_\nu,
\qquad\hbox{for $x\in W^\nu$,}
\label{LusztigS}
\end{equation}
where the sum can contain several $y\le x$ which have the same coset $yW_\nu$ (and this is how cancellation
can occur in the  sum \eqref{LusztigS}).
Since
$$T_x\mathbf{1}_\nu 
= (t^{\frac12})^{\ell(z)} T_{xz}\mathbf{1}_\nu,\ \hbox{for $z\in W_\nu$,}
$$
the coefficients $P^\nu_{y,x}$ in 
\begin{equation}C_x\mathbf{1}_\nu = \sum_{y\in W^\nu} (-1)^{\ell(x)-\ell(y)}
P_{y,x}^\nu T_y\mathbf{1}_\nu
\qquad\hbox{are}\qquad
P_{y,x}^\nu = \sum_{z\in W_\nu} (-1)^{\ell(y)-\ell(yz)} (t^{\frac12})^{\ell(z)} P_{yz,x}.
\label{singKLpolysdefn}
\end{equation}

Since $C_wT_{s_i}  = -t^{-\frac12} C_w$ unless $ws_i>w$ (see \cite[Prop.\ 7.14(a)]{Hu}), it follows that $C_w(T_{s_i}+t^{-\frac12})=0$ unless $ws_i>w$ so that
\begin{equation}\label{Cwvanishing}
C_w \mathbf{1}_\nu  = 0, \ \ \hbox{unless $w\in W^\nu$.}
\end{equation}
In summary, right multiplication by $\mathbf{1}_\nu$ is a surjective homomorphism of $\ZZ[t^{\frac12}, t^{-\frac12}]$-modules      
\begin{equation}
\begin{array}{ccll}
H &\longrightarrow &H\mathbf{1}_\nu  \\
T_w &\longmapsto &(t^{\frac12})^{\ell(v)}T_u\mathbf{1}_\nu ,  &\hbox{if $w=uv$ with $u\in W^\nu$ and $v\in W_\nu$, and } \\
C_w &\longmapsto &\begin{cases} C_w\mathbf{1}_\nu, &\hbox{if $w\in W^\nu$,} \\ 0, &\hbox{if $w\not\in W^\nu$.}
\end{cases} \\ \\
\end{array}
\label{singularmap}
\end{equation}

\subsubsection{Parabolic KL polynomials}\label{parabolicKLpolys}

As in \eqref{stabsandsubsets}, let $W_\gamma = \langle s_k\ |\ k\not\in \gamma\rangle$ and let
${}^\gamma W$ be the set of minimal length coset representatives of the 
cosets in $W_\gamma\backslash W$.  The $\ZZ[t^\frac12, t^{-\frac12}]$-module 
$$\hbox{$\varepsilon_\gamma H$ has basis }\quad
\{ \varepsilon_\gamma T_u\ |\ u\in {^\gamma W}\}
\qquad\hbox{and}\qquad
\overline{\phantom{T}}\colon \varepsilon_\gamma H \to \varepsilon_\gamma H
$$
since $\overline{\varepsilon_\gamma} = \varepsilon_\gamma$.
The Bruhat order ${}^\gamma W$ is the restriction of the Bruhat order on $W$ to ${}^\gamma W$ and,
with these structures, $\varepsilon_\gamma H$ is a KL-module.

Let $w_\gamma$ be the longest element of $W_\gamma$ and let $u\in {}^\gamma W$.
Since $T_{s_i}C_{w_\gamma u} = -t^{-\frac12}C_{w_\gamma u}$ for simple reflections $s_i\in W_\gamma$ (see \cite[Prop.\ 7.14(a)]{Hu}), it follows that
$C_{w_\gamma u}\in \varepsilon_\gamma H$.  Thus
\begin{equation}
\hbox{$\varepsilon_\gamma H$ has KL-basis }\quad
\{ C_{w_\gamma u}\ |\ u\in {}^\gamma W\},
\label{parabolicKLbasis}
\end{equation}
where $\{ C_w\ |\ w\in W\}$ is the KL-basis of $H$.  
In summary, there is an injective homomorphism of KL-modules
\begin{equation}
\begin{matrix}
\varepsilon_\gamma H &\longrightarrow &H \\
\varepsilon_\gamma T_u &\longmapsto & \varepsilon_\gamma T_u \\
C_{w_\gamma u} &\longmapsto &C_{w_\gamma u}
\end{matrix}
\label{parabolicmap}
\end{equation}
where $u\in {}^\gamma W$.

If $x\in {}^\gamma W$ then, from the second formula in \eqref{LusztigC},
\begin{equation}
C_{w_\gamma x} 
= \sum_{y\le w_\gamma x \atop y\in W} (-1)^{\ell(w_\gamma x)-\ell(y)}P_{y,w_\gamma x}(t^\frac12) T_y
=\sum_{w_\gamma y \le w_\gamma x\atop y\in {}^\gamma W}
(-1)^{\ell(w_\gamma x)-\ell(w_\gamma y)}P_{w_\gamma y,w_\gamma x}(t^\frac12) \varepsilon_\gamma T_y. 
\label{LusztigP}
\end{equation}
where, by the second formula in \eqref{projectors},
if $w\in W$ and $w=vu$ with $u\in {}^\gamma W$ and $v\in W_\gamma$ then
\begin{equation}
\varepsilon_\gamma T_w = \varepsilon_\gamma  T_v T_u = (-t^{-\frac12})^{\ell(v)}\varepsilon_\gamma T_u 
=(-t^{-\frac12})^{\ell(v)} \sum_{z\in W_\nu} (-t^{\frac12})^{\ell(w_\gamma)-\ell(z)} T_{zu}.
\label{parabolicexp}
\end{equation}

\subsubsection{Singular block parabolic KL polynomials}\label{singparabolicKLsection}

As in \eqref{stabsandsubsets},
$$\hbox{let}\quad W_\gamma = \langle s_k\ |\ k\not\in \gamma\rangle
\qquad\hbox{and let}\qquad W_\nu = \langle s_j\ |\ j\in J\rangle.
$$
Let $w_\gamma$ be the longest element of $W_\gamma$ and let
$\varepsilon_\gamma$ and $\mathbf{1}_\nu$ be as defined in \eqref{projectors}.
The composite of \eqref{singularmap} and \eqref{parabolicmap}
\begin{equation}
\begin{matrix}
\varepsilon_\gamma H &\longrightarrow &H &\longrightarrow &H\mathbf{1}_\nu \\
\varepsilon_\gamma T_u &\longmapsto & \varepsilon_\gamma T_u &\longmapsto &\varepsilon_\gamma T_u\mathbf{1}_\nu\\
C_{w_\gamma x} &\longmapsto &C_{w_\gamma x} &\longmapsto &C_{w_\gamma x}\mathbf{1}_\nu
\end{matrix}
\qquad\hbox{has image}\qquad
\varepsilon_\gamma H \mathbf{1}_\nu.
\label{singularparabolicmap}
\end{equation}
Let ${}^\gamma W$ be the set of minimal length coset representatives of the 
cosets in $W_\gamma\backslash W$, and let $W^\nu$ be the set of minimal length coset representatives of the 
cosets in $W/W_\nu$.   
From \eqref{Cwvanishing}, $C_w \mathbf{1}_\nu = 0$ unless $w\in W^\nu$, and so, in \eqref{singularparabolicmap},
$$\hbox{if $u\in {}^\gamma W$ then\quad $C_{w_\gamma u}\mathbf{1}_\nu = 0$ unless $w_\gamma u\in W^\nu$.}$$

By \cite[Ch.\ IV \S 1 Ex. 3]{Bou}),
the elements of ${}^\gamma W\cap W^\nu$ are the minimal length elements of the double cosets in $W_\gamma\backslash W/W_\nu$ 
and are a set of representatives of the double cosets in $W_\gamma\backslash W/W_\nu$.  If $W_\gamma a W_\nu$ is a 
double coset in $W_\gamma\backslash W /W_\nu$ then
there is a unique element $u\in W_\gamma a W_\nu$ of minimal length and 
\begin{equation}
\hbox{if
$w\in W_\gamma a W_\nu$\quad then}\quad 
\begin{array}{l}
\hbox{$w=vuz$, with $v\in W_\gamma$, $z\in W_\nu$} \\
\hbox{and $\ell(w) = \ell(v)+\ell(u)+\ell(z)$. }
\end{array}
\label{doublecosetindexing}
\end{equation}
Note that \eqref{doublecosetindexing} does \emph{not} imply that $\Card(W_\gamma a W_\nu) = \Card(W_\gamma)\Card(W_\nu)$.  

\begin{prop}\label{doublecosetrepbasis}  
Let $u\in {}^\gamma W \cap W^\nu$ so that $u$ is a minimal length element
of a double coset in $W_\gamma\backslash W/W_\nu$.

\item[(a)] If  $w_\gamma u\notin W^\nu$ then $\varepsilon_\gamma T_u \mathbf{1}_\nu = 0$.
\item[(b)] If  $w_\gamma u \in W^\nu$ then
\begin{equation}
\varepsilon_\gamma T_u \mathbf{1}_\nu = 
(-t^{\frac12})^{\ell(w_\gamma)}(t^{-\frac12})^{\ell(w_\nu)}
\sum_{v\in W_\gamma, z\in W_\nu} (-t^{-\frac12})^{\ell(v)}(t^{\frac12})^{\ell(z)} T_{vuz}.
\label{doublecosetexpansion}
\end{equation}
\end{prop}
\begin{proof}
The group $W_\gamma$ acts on the coset space $W/W_\nu$.  The coset space $W/W_\nu$ can always be identified with 
the orbit $W\nu$ for some element $\nu\in \fa^*$, where $\fa^* = \fa_\ZZ\otimes_\ZZ \RR$.  
Thus a $W_\gamma$ orbit is $W_\gamma\lambda$ for some $\lambda\in \fa^*$.
We may take $\lambda = u\nu$ where $u$ is minimal length in the orbit $W_\gamma u W_\nu$.
Let $W_\lambda = \mathrm{Stab}_W(\lambda)=uW_\nu u^{-1}$. Since the stabilizer of the $W_\gamma$ action on $\lambda$ is $W_\gamma\cap W_\lambda$,
the elements of the orbit $W_\gamma\lambda$ are indexed by the set $W_\gamma^\lambda$ of minimal length representatives of the cosets 
in $W_\gamma/(W_\gamma\cap W_\lambda)$.  It follows that 
$$W_\gamma uW_\nu = \{ xuy\ |\ x\in W_\gamma^\lambda, y\in W_\nu\}
\qquad\hbox{with}\quad \Card(W_\gamma u W_\nu) = \Card(W_\gamma^\lambda)\Card(W_\nu).$$

\smallskip\noindent
(a) 
 Assume $w_\gamma u\not\in W^\nu$.  Then there exists $s_i\in W_\gamma\cap W_\lambda$.
 So $s_iu = u s_j$ with $s_j\in W_\nu$ and it follows that
$$\varepsilon_\gamma T_u \mathbf{1}_\nu
= (-t^{\frac12})\varepsilon_\gamma T_{s_i}T_u \mathbf{1}_\nu
= (-t^{\frac12})\varepsilon_\gamma T_{s_i u} \mathbf{1}_\nu
= (-t^{\frac12})\varepsilon_\gamma T_{u s_j} \mathbf{1}_\nu
= (-t^{\frac12})\varepsilon_\gamma T_u T_{s_j} \mathbf{1}_\nu
= -t\varepsilon_\gamma T_u \mathbf{1}_\nu,
$$
giving that $\varepsilon_\gamma T_u \mathbf{1}_\nu = 0$.

\smallskip\noindent
(b)  Continuing from the proof of (a), 
$\varepsilon_\gamma T_u \mathbf{1}_\nu\ne 0$ 
only when $W_\gamma\cap W_\lambda = \{1\}$ so that
$$
W_\gamma^\lambda = W_\gamma, \qquad\hbox{in which case}\qquad
\Card(W_\gamma u W_\nu) = \Card(W_\gamma)\Card(W_\nu)\quad\hbox{and} 
$$
$$
W_\gamma uW_\nu = \{ xuy\ |\ x\in W_\gamma, y\in W_\nu\}
\quad\hbox{and}\quad w_\gamma u \in W^{\nu}.
$$
Then
\begin{align*}
\varepsilon_\gamma T_u \mathbf{1}_\nu
&= \left((-t^{\frac12})^{\ell(w_\gamma)}\sum_{v\in W_\gamma} (-t^{-\frac12})^{\ell(v)} T_v\right)
T_u \left((t^{-\frac12})^{\ell(w_\nu)}\sum_{z\in W_\nu} (t^{\frac12})^{\ell(z)} T_z\right) \\
&= (-t^{\frac12})^{\ell(w_\gamma)}
(t^{-\frac12})^{\ell(w_\nu)}\sum_{v\in W_\gamma, z\in W_\nu} (-t^{-\frac12})^{\ell(v)}(t^{\frac12})^{\ell(z)} T_vT_u T_z \\
&= (-t^{\frac12})^{\ell(w_\gamma)}
(t^{-\frac12})^{\ell(w_\nu)}\sum_{v\in W_\gamma, z\in W_\nu} (-t^{-\frac12})^{\ell(v)}(t^{\frac12})^{\ell(z)} T_{vuz},
\end{align*}
where the first equality follows from \eqref{projectors} and the third equality follows from \eqref{doublecosetindexing}.
\end{proof}

Since $\overline{\mathbf{1}_\nu} = \mathbf{1}_\nu$ and $\overline{\varepsilon_\gamma} = \varepsilon_\gamma$,
the restriction of $\overline{\phantom{T}}\colon H\to H$ provides
$$
\overline{\phantom{T}}\colon \varepsilon_\gamma H \mathbf{1}_\nu \to \varepsilon_\gamma H \mathbf{1}_\nu,
\qquad\hbox{and}\qquad
\varepsilon_\gamma H \mathbf{1}_\nu
\quad\hbox{has basis}\quad \{ \varepsilon_\gamma T_u\mathbf{1}_\nu\ |\ \hbox{$u\in {}^\gamma W$ and $w_\gamma u\in W^\nu$} \},
$$
and the restriction of the Bruhat order on $W$  provides a partial order on 
the set $\{ u\in {}^\gamma W\ |\ w_\gamma u\in W^\nu\}$.
With these structures,
$\varepsilon_\gamma H \mathbf{1}_\nu$ is a KL-module and,
from \eqref{singularKLbasis} and \eqref{parabolicKLbasis},
\begin{equation}
\varepsilon_\gamma H \mathbf{1}_\nu
\quad\hbox{has KL-basis}\quad \{ C_{w_\gamma u}\mathbf{1}_\nu\ |\ \hbox{$u\in {}^\gamma W$ and $w_\gamma u\in W^\nu$} \}
\label{singularparabolicKLbasis}
\end{equation}
and, using 
\eqref{LusztigP} and Proposition \ref{doublecosetrepbasis},
\begin{align}
C_{w_\gamma x}\mathbf{1}_\nu 
&= \sum_{w_\gamma y\le w_\gamma x\atop y\in {}^\gamma W} 
(-1)^{\ell(w_\gamma x)-\ell(y)} P_{w_\gamma y,w_\gamma x}(t^{\frac12}) \varepsilon_\gamma T_y\mathbf{1}_\nu 
\nonumber \\
&=\sum_{w_\gamma y\le w_\gamma x\atop y\in {}^\gamma W, w_\gamma y\in W^\nu}
(-1)^{\ell(w_\gamma x)-\ell(w_\gamma y)} P^\nu_{w_\gamma y,w_\gamma x}(t^{\frac12}) \varepsilon_\gamma T_y\mathbf{1}_\nu,
\label{LusztigSP}
\end{align}
where, as in \eqref{singKLpolysdefn},
$$P_{w_\gamma y, w_\gamma x}^\nu = \sum_{z\in W_\nu} (-1)^{\ell(w_\gamma y)-\ell(w_\gamma yz)}
P_{w_\gamma y z, w_\gamma x}.$$


\section{Decomposition numbers via Hecke algebras}

\subsection{Affine Kac-Moody and $\nu$ negative level rational}

With $W_0$ and $\fa^*_\ZZ$ as in \eqref{wtsWeylgpdefn}, let $\mathring{\fg}$ be a finite dimensional
complex reductive Lie algebra with Cartan subalgebra $\fa$ and Borel subalgebra
$\mathring{\fb}$ containing $\fa$ such that the Weyl group is $W_0$, the weight lattice is $\fa_\ZZ^*$ and
the simple coroots are $\alpha_1^\vee, \ldots, \alpha_n^\vee$.
Let $\fg$ be the corresponding affine Kac-Moody Lie algebra (see \cite[(7.2.2)]{Kac}),
\begin{equation}
\fg = (\mathring{\fg}\otimes_\CC \CC[\epsilon,\epsilon^{-1}])\oplus\CC K \oplus \CC d, 
\qquad\hbox{with Cartan subalgebra}\quad \fh = \fa\oplus \CC K \oplus \CC d
\label{gAffineandCartan}
\end{equation}
and positive real roots $R_{\mathrm{re}}^+$ and integral weight lattice $\fh_\ZZ^*$.  Let
$\alpha_0^\vee, \alpha_1^\vee,\ldots, \alpha_n^\vee$ be the simple coroots of $\fg$ with respect
to the Borel subalgebra 
$\fb = \mathring{\fb}\oplus \CC K \oplus \CC d \oplus (\mathring{\fg}\otimes_\CC \epsilon\CC[\epsilon])$ (see
\cite[Theorem 7.4]{Kac}) and let
$$\hat\rho\in \fh^*
\qquad \hbox{such that $\langle \hat\rho, \alpha_i^\vee\rangle=1$, for $i\in \{0, 1,\ldots, n\}$}$$
(see \cite[(6.2.8) and (12.4.3)]{Kac}).  For $\nu\in \fh^*$ define
\begin{equation}
\Delta^+(\nu) = \{ \alpha\in R_{\mathrm{re}}^+\ |\ \langle \nu+\hat\rho,\alpha^\vee\rangle\in \ZZ\} 
\qquad\hbox{and}\qquad
W(\nu) =\langle s_\alpha\ |\ \alpha\in \Delta^+(\nu)\rangle
\label{Wnudefn}
\end{equation}
and define the \emph{dot action of $W$ on $\fh^*$} by
\begin{equation}
w\circ\lambda = w(\lambda+\hat\rho)-\hat\rho,
\qquad\hbox{for $w\in W$ and $\lambda\in \fh^*$.}
\label{affinedotaction}
\end{equation}
If $\nu\in \fh_\ZZ^*$ then $W(\nu) = \langle s_\alpha\ |\ \alpha\in R_{\mathrm{re}}^+\rangle = W$
as defined in  \eqref{Wnudefn} is the full affine Weyl group.

For $\lambda\in \fh^*$ let
\begin{equation}
\begin{array}{ll}
M(\lambda)\qquad &\hbox{be the Verma module of highest weight $\lambda$ for $\fg$, and} \\
L(\lambda) &\hbox{the irreducible module of highest weight $\lambda$ for $\fg$}.
\end{array}
\label{Vermaforg}
\end{equation}

A weight $\nu\in \fh^*$ is \emph{negative level rational} if $\nu$ satisfies:
\begin{enumerate}
\item[(a)] (negativity/antidominance) 
If $i\in\{0,1, \ldots, n\}$ then
$\langle \nu+\hat\rho,\alpha_i^{\vee} \rangle\in \QQ_{\leq 0}$,
\item[(b)] (negative level) $\langle \nu+\hat\rho, K\rangle\in \QQ_{<0}$.
\end{enumerate}
Given condition (a) the only additional content of (b) is that $\langle \nu+\hat\rho, K\rangle\ne 0$,
(see the statement of \cite[Theorem 3.3.6]{KT96}).

\begin{thm} \label{KTnegativelevel} \emph{\cite[Theorem 0.1]{KT96}}
Let $\fg$ be an affine Kac-Moody Lie algebra and let
$\nu\in \fh^*$ be negative level rational.
Let $w\in W$ be of minimal length in $wW(\nu)$.
Letting $<$ denote the Bruhat order on $W$, let $x\in W(\nu)$ be such that  
\begin{equation}
\hbox{if $w'\in W$ and $w'<wx$}\qquad 
\hbox{then}\quad\hbox{$w'\circ \nu\neq wx\circ\nu$.}
\label{wallcondition}
\end{equation}
Let $\mathrm{ch}(M)$ denote the character (weight space generating function) of a $\fg$-module $M$.  Then
\begin{equation}
\mathrm{ch}(L(wx\circ \nu))
=\sum_{y\leq _{\nu}x} (-1)^{\ell_{\nu}(x)-\ell_{\nu}(y)}P_{y,x}^{\nu}(1)
\mathrm{ch}(M(wy\circ \nu)),
\label{KTneglevelLintermsofM}
\end{equation}
where $\ell_\nu$ is the length function, $\le_\nu$ is the Bruhat order and 
$P_{y,x}^\nu$ are the Kazhdan-Lusztig polynomials (see \eqref{LusztigC})
for the Coxeter group $W(\nu)$, and the sum is over $y\in W(\nu)$ such that
$y\le_\nu x$.
\end{thm}
\noindent
This statement generalizes a conjecture of Lusztig
\cite[Conj. 2.5c]{Lu90}, proved by Kashiwara-Tanisaki in \cite{KT95}.  It is 
a negative level affine version of the original ``Kazhdan-Lusztig conjecture'' 
of \cite[Conjecture 1.5]{KL79}.  A refinement of \cite[Conjecture 1.5]{KL79} is the Jantzen 
conjecture, which was proved by Beilinson-Bernstein \cite[Cor. 5.3.5]{BB}.  The ``Jantzen conjecture'' 
result generalizes to the negative level affine setting, as proved by Shan \cite[Proposition 5.5 and Theorem 6.4]{Sh}.

\subsection{The Kashiwara-Tanisaki theorem in Hecke algebra notation}\label{KTtoHeckemodules} 

The purpose of this subsection is to repackage the result of Theorem \ref{KTnegativelevel} (in the strong ``Jantzen conjecture'' form)
into the Hecke algebra notations of Section \ref{HeckeKLsection}.

Keep the notations of Theorem \ref{KTnegativelevel} so that $\fg$ is the affine Lie algebra, $\fh$ is the Cartan subalgera
as in \eqref{gAffineandCartan} and $\nu\in \fh^*$ is negative level rational.
$$\hbox{Let $H$ be the Hecke algebra of the group $W(\nu)$,}$$
where $W(\nu)$ is as defined in \eqref{Wnudefn} and $H$ is as defined in \eqref{eqnT2}.  Notice that we drop the dependence on $\nu$ in the notation for the Hecke agebra. From now on, we will also drop the dependence on $\nu$ in the Coxeter group and write $W$ instead of $W(\nu)$, to simplify the notation.
Let
$$\hbox{$K(\cO[\nu])$ be the free $\ZZ[t^{\frac12}, t^{-\frac12}]$-module generated by symbols
$[M(x\circ\nu)]$}
$$
for $x\in W^\nu$,  the set of minimal representatives of $W_\nu$-cosets (where  $W_\nu = \mathrm{Stab}(\nu)$ is the stabilizer of $\nu$ in $W$ under the dot action).  Define elements $[L(y\circ\nu)]$, $y\in W^\nu$, by the equation
$$[M(x\circ\nu)] = 
\sum_{y\le x} 
\left(\sum_{i\in \ZZ_{\ge 0}} \left[ \frac{M^{(i)}(x\circ\nu)}{M^{(i-1)}(x\circ\nu)} :
L(y\circ\nu)\right] (t^{\frac12})^i \right) [L(y\circ\nu)],$$
where $[M:L(\mu)]$ denotes the multiplicity of the simple $\fg$-module $L(\mu)$ of highest weight
$\mu$ in a composition series of $M$ and
$$M(\lambda)=M(\lambda)^{(0)} \supseteq M(\lambda)^{(1)}\supseteq
\cdots
\qquad\hbox{is the Jantzen filtration of $M(\lambda)$,}
$$
see, for example, \cite[(2.5)]{OR}.

\medskip\noindent
\emph{Case R: regular $\nu$.}
Let $\nu\in \fh^*$ such that $\langle \nu+\hat\rho, \alpha_i^\vee\rangle\in \QQ_{<0}$.
Then $\mathrm{Stab}(\nu)=\{1\}$ under the dot action of \eqref{affinedotaction}.
In this case the strong ``Jantzen conjecture'' version of Theorem \ref{KTnegativelevel} 
(see  \cite[Theorem 6.4 and Proposition 5.5]{Sh}) is equivalent to a $\ZZ[t^{\frac12}, t^{-\frac12}]$-module isomorphism
\begin{equation}
\begin{matrix}
K(\cO[\nu]) &\mapright{\sim} &H \\
[M(y\circ \nu)] &\longmapsto &T_y \\
[L(x\circ \nu)] &\longrightarrow &C_x
\end{matrix}
\qquad\qquad\hbox{where $T_y$ and $C_x$ are as in \eqref{LusztigC}.}
\label{R}
\end{equation}

\medskip\noindent
\emph{Case S: singular $\nu$.}
Let $\nu\in \fh^*$ such that $\langle \nu+\hat\rho, \alpha_i^\vee\rangle\in \QQ_{\le0}$ and let 
$$J=\{ j\in \{0,1,\ldots, n\}\ |< \langle \nu+\hat\rho, \alpha_j^\vee\rangle = 0\}
\qquad\hbox{so that}\qquad
W_\nu = \langle s_j\ |\ j\in J\rangle$$
is the stabilizer of the dot action of $W$ on $\nu$.  
Let $\mathbf{1}_\nu$ be the element of $H$ defined in \eqref{projectors}.
Then the strong ``Jantzen conjecture" version of Theorem \ref{KTnegativelevel} 
(see  \cite[Theorem 6.4 and Proposition 5.5]{Sh}) is equivalent to a $\ZZ[t^{\frac12}, t^{-\frac12}]$-module isomorphism
\begin{equation}
\begin{matrix}
K(\cO[\nu]) &\mapright{\sim} &H\mathbf{1}_\nu \\
[M(y\circ \nu)] &\longmapsto &T_y\mathbf{1}_\nu \\
[L(x\circ \nu)] &\longrightarrow &C_x\mathbf{1}_\nu
\end{matrix}
\qquad\qquad\hbox{where $T_y\mathbf{1}_\nu$ and $C_x\mathbf{1}_\nu$ are as in \eqref{LusztigS}.}
\label{S}
\end{equation}

%
%

\subsection{Decomposition numbers for parabolic $\cO$}

Keep the notations for the affine Lie algebra as in \eqref{gAffineandCartan},
and let $e_0, \ldots, e_n$, $f_0, \ldots, f_n$, $\fa$ and $d$
be Kac-Moody generators for $\fg$. Let
$\gamma\subseteq \{0, 1, \ldots, n\}$ with $\gamma\ne \emptyset$
and define, following \cite[\S7]{Soe98}, a $\ZZ$-grading on $\fg$ by
$\deg(d) = 0$,
$\deg(h)=0\ \hbox{for $h\in \fa$,}$
$$\deg(e_i) = \begin{cases} 0, &\hbox{if $i\in \gamma$,} \\
1, &\hbox{if $i\not\in \gamma$,}
\end{cases}
\qquad\hbox{and}\qquad
\deg(f_i) = \begin{cases} 0, &\hbox{if $i\in \gamma$,} \\
-1, &\hbox{if $i\not\in \gamma$.}
\end{cases}
$$
Let
\begin{equation}
\fg_\gamma = \{ x\in \fg\ |\ \deg(x)=0\}
\qquad\hbox{and}\qquad
\fb_\gamma = \{ x\in \fg\ |\ \deg(x)\ge 0\}.
\label{levidefn}
\end{equation}
Following the first two paragraphs of \cite[\S3]{Soe98}, the
\emph{parabolic category $\cO$ (with respect to $\deg$)} is the category $\cO_{\fg_\gamma}^\fg$
of $\fg$-modules $M$ such that
\begin{enumerate}
\item[(a)] $M$ is $\fg_\gamma$-semisimple,
\item[(b)] $M$ is $\fb_\gamma$-locally finite, i.e. If $m\in M$ then 
$\dim(U\fb_\gamma\cdot m) < \infty$.
\end{enumerate}
Let $(\fa^*)_\gamma^+$ be an index set for the 
finite dimensional simple $\fg_\gamma$-modules
$\{ L_{\fg_\gamma}(\lambda)\ |\ \lambda\in (\fa^*)_\gamma^+\}$.
The \emph{standard modules in $\cO_{\fg_\gamma}^\fg$} are
\begin{equation}
\Delta_{\fg_\gamma}^{\fg}(\lambda) = U\fg\otimes_{U\fb_\gamma} L_{\fg_\gamma}(\lambda),
\qquad\hbox{for $\lambda\in (\fa^*)^+_\gamma$,}
\label{ParStdmodules}
\end{equation}
where $L_{\fg_\gamma}(\lambda)$ becomes a $\fb_\gamma$-module by setting 
$x n = 0$ if $n\in L_{\fg_\gamma}(\lambda)$ and $x\in \fg$ is homogeneous with $\deg(x)>0$.
The simple modules in $\cO_{\fg_\gamma}^\fg$ are the quotients
$$L(\lambda) = \frac{\Delta_{\fg_\gamma}^{\fg}(\lambda)}{(\hbox{max. proper submodule})},
\qquad\hbox{for $\lambda\in (\fa^*)_\gamma^+$.}
$$

Let $W_\gamma$ be the Weyl group corresponding to $\gamma$ as in \eqref{stabsandsubsets}.
Since $\gamma\ne \emptyset$ and $\fg$ is an affine Kac-Moody Lie algebra, the Lie
algebra $\fg_\gamma$ is finite dimensional and the integrable simple module 
$L_{\fg_\gamma}(\lambda)$ for the Lie algebra $\fg_\gamma$ has a BGG-resolution
(see \cite[Ex.\ 7.8.14]{Dx}),
\begin{equation}
0\longrightarrow \Delta_{\mathfrak{b}}^{\fg_\gamma}(w_\gamma\circ\lambda)
\longrightarrow \cdots \longrightarrow 
\bigoplus_{\stackrel{z\in W_\gamma}{\ell(z) = j}} \Delta_{\fb}^{\fg_\gamma}(z\circ \lambda)
\longrightarrow \cdots \longrightarrow
\Delta_{\fb}^{\fg_\gamma}(\lambda) \longrightarrow L_{\fg_\gamma}(\lambda) \longrightarrow 0.
\label{BGGres}
\end{equation}
where $\Delta_\fb^{\fg_\gamma}(\mu)$ denotes the Verma module of highest weight $\mu$ for $\fg_\gamma$
and $w_\gamma$ is the longest element of $W_\gamma$ (since $\gamma\ne \emptyset$ then
$W_\gamma$ is a finite Coxeter group and $w_\gamma$ exists).
(The dot action in \eqref{BGGres} coincides with the dot action defined in \eqref{affinedotaction} since
$W_\gamma$ is generated by $\{ s_i\ |\ i\not\in \gamma\}$ and both actions satisfy
$s_i\circ \lambda = \lambda - (\langle \lambda,\alpha_i^\vee\rangle +1)\alpha_i$ for
$i\not\in \gamma$.) 

As in \cite[paragraph before Prop.\ 7.5]{Soe98},
parabolic induction of the resolution \eqref{BGGres} to $\fg$ gives
\begin{equation}
0\longrightarrow \Delta_{\mathfrak{b}}^{\fg}(w_\gamma\circ\lambda)
\longrightarrow \cdots \longrightarrow 
\bigoplus_{\stackrel{z\in W_\gamma}{\ell(z) = j}} \Delta_{\fb}^{\fg}(z\circ \lambda)
\longrightarrow \cdots \longrightarrow
\Delta_{\fb}^{\fg}(\lambda) \longrightarrow \Delta_{\fg_\gamma}^\fg(\lambda) \longrightarrow 0,
\label{inducedBGGres}
\end{equation}
where $\Delta_{\fb}^{\fg}(\mu)=M(\mu)$ is the Verma module for $\fg$ as in \eqref{Vermaforg}.  
Thus the multiplicity of a simple $\fg$-module $L(\mu)$ in the
standard module $\Delta_{\fg_\gamma}^\fg(\lambda)$ is
\begin{equation}
[\Delta_{\fg_\gamma}^\fg(\lambda): L(\mu)] = \sum_{z\in W_\gamma} (-1)^{\ell(z)} 
[ \Delta_{\fb}^{\fg}(z\circ \lambda) : L(\mu)].
\label{BGGonchars}
\end{equation}
In the correspondence to the Hecke algebra as in \eqref{S},
$$\begin{matrix}
[M(zy\circ\nu)] = [\Delta_{\fb}^\fg(zy\circ\nu)] &\mapsto &T_{zy} \mathbf{1}_\nu \\
[\Delta_{\fg_\gamma}^\fg(w_\gamma y\circ\nu)] &\mapsto &\varepsilon_\gamma T_y \mathbf{1}_\nu \\
[L(w\circ\nu)] &\mapsto &C_w \mathbf{1}_\nu
\end{matrix}
$$
so that the identity in \eqref{BGGonchars} (which comes from the BGG resolution)
corresponds to the Hecke algebra identity (see \eqref{doublecosetexpansion})
$$\varepsilon_\gamma T_x\mathbf{1}_\nu 
= \sum_{z\in W_\gamma} (-t^{\frac12})^{\ell(w_\gamma)-\ell(z)} T_{zx}\mathbf{1}_\nu,
\qquad\hbox{
where $w_\gamma$ is the longest element of $W_\gamma$.}
$$

Let ${}^\gamma W$ be the set of minimal length representatives of cosets 
in $W_\gamma\backslash W$.
Let 
$$\hbox{$K(\cO_{\fg_\gamma}^\fg[w_\gamma\circ\nu])$ 
be the free $\ZZ[t^{\frac12}, t^{-\frac12}]$-module generated by symbols
$[\Delta_{\fg_\gamma}^\fg(w_\gamma x\circ\nu)]$,}
$$ 
for $x\in {}^\gamma W$ such that $w_\gamma x\in W^\nu$.  
Define elements $[L(w_\gamma y\circ\nu)]$, 
for $y\in {}^\gamma W$ such that $w_\gamma y\in W^\nu$, by the equation
$$[\Delta_{\fg_\gamma}^\fg(w_\gamma x\circ\nu)] = 
\sum_{y\le x} 
\left(\sum_{i\in \ZZ_{\ge 0}} \left[ \frac{\Delta_{\fg_\gamma}^\fg(w_\gamma x\circ\nu)^{(i)}}
{\Delta_{\fg_\gamma}^\fg(w_\gamma x\circ\nu)^{(i+1)} } :
L(w_\gamma y\circ\nu)\right] (t^{\frac12})^i \right) [L(w_\gamma y\circ\nu)],$$
where $[M:L(\mu)]$ denotes the multiplicity of the simple $\fg$-module $L(\mu)$ of highest weight
$\mu$ in a composition series of $M$ and
$$\Delta_{\fg_\gamma}^\fg(\lambda)=\Delta_{\fg_\gamma}^\fg(\lambda)^{(0)} 
\supseteq \Delta_{\fg_\gamma}^\fg(\lambda)^{(1)}\supseteq
\cdots
\qquad\hbox{is the Jantzen filtration of $\Delta_{\fg_\gamma}^\fg(\lambda)$}
$$
(see, for example, \cite[\S1.4, \S2.3 and \S2.10]{Sh} for the Jantzen filtration in this context).

\medskip\noindent
\emph{Case PR: Parabolic $\cO$, regular $\nu$.}
Let $\nu\in \fh^*$ such that $\langle \nu+\rho, \alpha_i\rangle\in \QQ_{<0}$.
Let $\gamma\subseteq \{0,1,\ldots, n\}$ and 
let $\fg_\gamma$ be the corresponding ``standard'' Levi subalgebra of $\fg$ as
defined in \eqref{levidefn}
with Weyl group $W_\gamma = \langle s_k\ |\ k\in \gamma\rangle$
as defined in \eqref{stabsandsubsets}.
Then Theorem \ref{KTnegativelevel}  (or \eqref{S}) combined with \eqref{BGGonchars} and \eqref{parabolicmap}
is equivalent, in the strong ``Jantzen conjecture'' form (see  \cite[Theorem 6.4 and Proposition 5.5]{Sh}) to a 
$\ZZ[t^{\frac12}, t^{-\frac12}]$-module isomorphism
\begin{equation}
\begin{matrix}
K(\cO_{\fg_\gamma}^{\fg}[w_\gamma\circ \nu]) &\mapright{\sim} &\varepsilon_\gamma H \\
[\Delta_{\fg_\gamma}^\fg(w_\gamma y\circ \nu)] &\longmapsto &\varepsilon_\gamma T_y \\
[L(w_\gamma x\circ \nu)] &\longrightarrow &C_{w_\gamma x}
\end{matrix}
\label{PR}
\end{equation}

\medskip\noindent
\emph{Case PS: parabolic $\cO$, singular $\nu$.}
Let $\nu\in \fh^*$ such that $\langle \nu+\rho, \alpha_i\rangle\in \QQ_{\le 0}$.
The maps in \eqref{S} and \eqref{PR}
can be packaged into a single statement as follows:
If $\nu\in \fh^*$ is such that $\langle \nu+\rho, \alpha_i\rangle\in \QQ_{\le 0}$ 
and $W_\nu = \mathrm{Stab}(\nu)$ is the stabilizer of $\nu$ in $W$ under the dot action then
\begin{equation}
\begin{matrix}
K(\cO_{\fg_\gamma}^{\fg}[w_\gamma\circ\nu]) &\mapright{\sim} &\varepsilon_\gamma H \mathbf{1}_\nu \\
[\Delta_{\fg_\gamma}^\fg(w_\gamma y\circ \nu)] 
&\longmapsto &\varepsilon_\gamma T_y\mathbf{1}_\nu \\
[L(w_\gamma x\circ \nu)] &\longrightarrow &C_{w_\gamma x}\mathbf{1}_\nu
\end{matrix}
\label{PS}
\end{equation}


\subsection{Decomposition numbers for quantum groups}

In 1989 and 1990, Lusztig made conjectures that the decomposition numbers for
representations of quantum groups can be picked up by Kazhdan-Lusztig polynomials for the affine Weyl goup.
Let $q\in \CC^\times$ and let $U_q(\mathring{\fg})$ be the Drinfel'd-Jimbo quantum group corresponding to $\mathring{\fg}$.  Let
\begin{align*}
&M_q(\lambda)\quad &\hbox{the Verma module of highest weight $\lambda$ for $U_q(\mathring{\fg})$,} \\
&\Delta_q(\lambda) &\hbox{the Weyl module for $U_q(\mathring{\fg})$ of highest weight $\lambda$,} \\
&L_q(\lambda) &\hbox{the simple module for $U_q(\mathring{\fg})$ of highest weight $\lambda$,}
\end{align*}
the conjectures \cite[Conj. 2.5]{Lu90} and \cite[Conj. 8.2]{Lu89} are
\begin{align*}
L_q(x\circ\nu) &= \sum_{y\in W_0\atop y\le x} (-1)^{\ell(v)+\ell(w)}P_{y,x}(1) M_q(y\circ\nu), 
&\hbox{if $q=1$ or $q^2$ is not a root of unity,} \\
L_q(x\circ\nu) &= \sum_{y\in W\atop y\le x} (-1)^{\ell(v)+\ell(w)}P_{y,x}(1) M_q(y\circ\nu), 
&\hbox{if $q^2$ is a primitive $\ell$-th root of unity,} \\
L_q(x\circ\nu) &= \sum_{y\in W\atop y\le x} (-1)^{\ell(v)+\ell(w)}P_{y,x}(1) \Delta_q(y\circ\nu), 
&\hbox{if $q^2$ is a primitive $\ell$-th root of unity,} 
\end{align*}
where, with $h$ and $\varphi^\vee$ as in \eqref{phicheckdefn} below, $\nu$ is an element of 
\begin{equation}
A_{-\ell-h} = \{ \nu\in \fa^*_\ZZ\ |\ 
\hbox{$\langle \nu,\varphi^\vee\rangle \ge -\ell-1$ and
$\langle \nu, \alpha_i^\vee\rangle \le -1$ for $i\in \{1, \ldots, n\}$}\}.
\label{ellAlcovefirst}
\end{equation}
These conjectures motivated Theorem \ref{negleveltoquantum} below \cite[Theorem 38.1]{KL94}
which had been previously conjectured by Lusztig \cite[Conjecture 2.3]{Lu90}.

Theorem \ref{negleveltoquantum} provides a connection between the representations of affine Lie algebras
and the representations of quantum groups.  Let us first sketch this relation on the level of weights.
Keep the notation for affine Lie algebras as in \eqref{gAffineandCartan}-\eqref{affinedotaction}.
Following \cite[\S6.2]{Kac} and coordinatizing $\fh^* = \CC\Lambda_0+\fa^*+\CC\delta$ with
$\langle \fa^*, K\rangle = 0$, $\langle \fa^*, d\rangle=0$, 
$$\langle \Lambda_0, K\rangle = 1,\quad
\langle \Lambda_0, \fa\rangle = 0, \quad
\langle \Lambda_0, d\rangle = 0, \quad
\langle \delta, K \rangle =0, \quad
\langle \delta, \fa\rangle = 0, \quad
\langle \delta, d\rangle = 1,
$$
then $\varphi^\vee\in \fa$ and the \emph{dual Coxeter number} $h$ are such that
\begin{equation}
\alpha_0^\vee = -\varphi^\vee+K,
\qquad\hbox{and}\qquad
\hat\rho = \rho+h\Lambda_0,
\label{phicheckdefn}
\end{equation}
where $\hat\rho\in \fh^*$ and $\rho\in \fa^*$ are as in \eqref{affinedotaction} and \eqref{dotaction}, respectively.
Let $\ell\in \ZZ_{>0}$.
Weights of $\fg$-modules that are level $-\ell-h$ are elements of $(-\ell-h)\Lambda_0+\fa^*+\CC\delta$.
Restricting modules in $\cO_{\fg_0}^\fg$ to the subalgebra $\fg'=[\fg,\fg]$ loses the information of $\CC\delta$, 
and in the diagram
\begin{equation}
\begin{matrix}
(-\ell-h)\Lambda_0+\fa^*+\CC\delta &\longrightarrow
&(-\ell-h)\Lambda_0+\fa^* &\longleftrightarrow &\fa^* \\
(-\ell-h)\Lambda_0+\lambda + a\delta &\longmapsto
&(-\ell-h)\Lambda_0+\lambda &\longmapsto &\lambda
\end{matrix}
\label{afftofinproj}
\end{equation}
the second map is a bijection.
Using the definition of negative level rational from just before Theorem \ref{KTnegativelevel},
\begin{align*}
\{ &\nu\in \fh^*_\ZZ\ |\ \hbox{$\nu+\hat\rho$ is level $-\ell$ and $\nu$ is
negative level rational}\} \\
&= \{ \nu\in \fh^*_\ZZ\ |\ 
\hbox{$\langle \nu+\hat\rho, K\rangle = -\ell$ and
$\langle \nu+\hat\rho, \alpha_i^\vee\rangle \in \QQ_{\le 0}$ for $i\in \{0, \ldots, n\}$}\} \\
&= (-\ell-h)\Lambda_0+\{ \nu\in \fa_\ZZ\ |\ 
\hbox{
$\langle \nu-\ell\Lambda_0+\rho, \alpha_i^\vee\rangle \in \QQ_{\le 0}$ for $i\in \{0, \ldots, n\}$}\} \\
&= (-\ell-h)\Lambda_0+\{ \nu\in \fa_\ZZ\ |\ 
\hbox{$\langle \nu-\ell\Lambda_0+\rho,-\varphi^\vee+K\rangle \in \QQ_{\le 0}$ and
$\langle \nu+\rho, \alpha_i^\vee\rangle \in \QQ_{\le 0}$ for $i\in \{1, \ldots, n\}$}\} \\
&= (-\ell-h)\Lambda_0+\{ \nu\in \fa_\ZZ\ |\ 
\hbox{$\langle \nu,\varphi^\vee\rangle \ge -\ell-1$ and
$\langle \nu, \alpha_i^\vee\rangle \le -1$ for $i\in \{1, \ldots, n\}$}\} \\
&=(-\ell-h)\Lambda_0+A_{-\ell-h},
\end{align*}
and, in light of Theorem \ref{negleveltoquantum} below, the ``source'' of the alcove $A_{-\ell-h}$ in
\eqref{ellAlcovefirst} is the
negative level rational condition for weights of the affine Lie algebra.

Next we compare the 
dot action from \eqref{affinedotaction} to the dot action from \eqref{dotaction}.  
Following \cite[(6.5.2)]{Kac}, the action of a 
translation $t_\mu$ on $\fh^* = \CC\delta+\fa^*+\CC\Lambda_0$ is given by
\begin{align*}
t_{\mu}(a\delta+\lambda+m\Lambda_0)
&=\big(a-\langle \lambda, \mu\rangle
-\hbox{$\frac12$}m\langle \mu, \mu\rangle\big)\delta
+\lambda + m\mu + m\Lambda_0, \quad\hbox{and} \\
w(a\delta+\lambda+m\Lambda_0) &= a\delta + w\lambda + m\Lambda_0,
\ \ \hbox{for $w\in W_0$, the finite Weyl group.}
\end{align*}
Thus, if $\lambda\in \fa^*$ then
\begin{align*}
(t_\mu w)\circ (\lambda+(-\ell-h)\Lambda_0) 
&= (t_\mu w)(\lambda+(-\ell-h)\Lambda_0 + \hat\rho) - \hat\rho \\
&= (t_\mu w)(\lambda+(-\ell-h)\Lambda_0 + \rho+h\Lambda_0) - (\rho+h\Lambda_0) \\
&= t_\mu( w(\lambda+\rho)-\ell\Lambda_0) - \rho - h\Lambda_0 \\
&=( w(\lambda+\rho)-\ell\Lambda_0-\ell\mu) - \rho - h\Lambda_0 \bmod \delta \\
&=(w\circ \lambda) - \ell\mu +(-\ell-h)\Lambda_0 \bmod \delta,
\end{align*}
where it is important to note that that the $\circ$ on the
left side of this equation is the dot action of  \eqref{affinedotaction} and the $\circ$ on the
right hand side is the dot action of \eqref{dotaction}.  This computation is the basis for using
\eqref{afftofinproj} to obtain an action of the \emph{affine} Weyl group $W$ on $\fa^*$ and define
the \emph{level $(-\ell-h)$ dot action of $W$} on $\fa^*$ by
\begin{equation}
(t_\mu w)\circ\lambda = (w\circ\lambda) - \ell \mu = w(\lambda+\rho) - \rho - \ell\mu,
\label{levelelldotactionfirst}
\end{equation}

Now let us state the Kazhdan-Lusztig theorem relating representations of affine Lie algebras
to representations of quantum groups at root of unity.
Let
$$\fg' = [\fg, \fg] = \mathring{\fg}\otimes_\CC \CC[\epsilon,\epsilon^{-1}] + \CC K.$$
In the context of \eqref{stabsandsubsets} and \eqref{levidefn}, let $\gamma = \{0\}$ so that 
$$\fg_\gamma = \fg_0 =\mathring{\fg}
\qquad\hbox{and}\qquad
\varepsilon_\gamma  = \varepsilon_0 = \sum_{w\in W_0} (-t^{\frac12})^{\ell(w_0)-\ell(z)}T_z,
$$
where $w_0$ is the longest element of $W_0$, the Weyl group of $\mathring \fg$.
By restriction, the modules in $\cO_{\fg_0}^\fg$ are $\fg'$-modules.
\begin{thm} \label{negleveltoquantum}
\cite[Theorem 38.1]{KL94}
There is an equivalence of categories
\begin{equation*}
\begin{matrix}
\left\{
\begin{matrix}
\hbox{finite length $\fg'$-modules} \\
\hbox{of level $-\ell-h$ in $\cO_{\fg_0}^\fg$} 
\end{matrix}
\right\}
&\stackrel{\sim}{\longleftrightarrow}
&\left\{
\begin{matrix}
\hbox{finite dimensional $U_q(\mathring{\fg})$-modules}\\
\hbox{with $q^{2\ell}=1$}
\end{matrix}
\right\}
\\
\Delta_{\fg_0}^\fg((-\ell-h)\Lambda_0+\lambda)
&\longmapsto
&\Delta_q(\lambda) \\
L((-\ell-h)\Lambda_0+\lambda)
&\longmapsto
&L_q(\lambda)
\end{matrix}
\end{equation*}
\end{thm}

\noindent
This statement of Theorem \ref{negleveltoquantum} is for the simply-laced (symmetric) case.  With the proper 
modifications to this statement the result holds for non-simply laced
cases as well, see \cite[\S 8.4]{Lu94} and \cite{Lu95}.

Let 
$$\hbox{$K(\hbox{fd$U_q(\mathring{\fg})$-mod})$ 
be the free $\ZZ[t^{\frac12},t^{-\frac12}]$-module generated by symbols
$[\Delta_q(\lambda)]$,}
$$ 
for $\lambda\in \fa_\ZZ^*$.
Define elements $[L_q(w_0 y\circ\nu)]$, 
for $\nu\in A_{-\ell-h}$ and $y\in {}^0 W$ such that $w_0 y\in W^\nu$, by the equation
$$[\Delta_q(w_0 x\circ\nu)] = 
\sum_{y\le x} 
\left(\sum_{i\in \ZZ_{\ge 0}} \left[ \frac{\Delta_q(w_0 x\circ\nu)^{(i)}}
{\Delta_q(w_0 x\circ\nu)^{(i+1)} } :
L_q(w_0 y\circ\nu)\right] (t^{\frac12})^i\right) [L_q(w_0 y\circ\nu)],$$
where $[M:L_q(\mu)]$ denotes the multiplicity of the simple $\fg$-module 
$L_q(\mu)$ of highest weight
$\mu$ in a composition series of $M$ and
$$\Delta_q(\lambda)=\Delta_q(\lambda)^{(0)} 
\supseteq \Delta_q(\lambda)^{(1)}\supseteq
\cdots
\qquad\hbox{is the Jantzen filtration of $\Delta_q(\lambda)$}
$$
(see, for example, \cite[\S1.4, \S2.3 and \S2.10 and Cor.\ 2.14]{Sh}  and \cite[\S4]{JM} for the Jantzen filtration in this context).

\medskip\noindent
\emph{Case QG: quantum groups, integral weights.}
The maps in \eqref{PS} combined with the result of Theorem \ref{negleveltoquantum} 
can be packaged in terms of the affine Hecke algebra as follows:
Let $\nu\in A_{-\ell-h}$ and let $W_\nu = \mathrm{Stab}(\nu)$ 
is the stabilizer of $\nu$ in $W$ under the level $-\ell-h$ dot action.  Then
\begin{equation}
\begin{matrix}
K(\hbox{fd$U_q(\mathring{\fg})$mod}) &\mapright{\sim} 
&\displaystyle{ \bigoplus_{\nu\in A_{-\ell-h}} \varepsilon_0 H \mathbf{1}_\nu }\\
[\Delta_q(w_0 y\circ \nu)] 
&\longmapsto &\varepsilon_0 T_y\mathbf{1}_\nu \\
[L_q(w_0 x\circ \nu)] &\longrightarrow &C_{w_0 x}\mathbf{1}_\nu
\end{matrix}
\label{QG}
\end{equation}

\section{The Fock space Hecke KL-module in the general setting}



Keep the notation for the finite Weyl group $W_0$, the simple reflections $s_1, \ldots, s_n$ 
and the weight lattice $\fa_\ZZ^*$ as in \eqref{wtsWeylgpdefn}.  The \emph{affine Weyl group} is
\begin{equation}
W = \{ t_\mu w\ |\ \mu\in \fa_\ZZ^*, w\in W_0\}, \qquad\hbox{with}\qquad
t_\mu t_\nu = t_{\mu+\nu}, \quad\hbox{and}\quad
wt_\mu = t_{w\mu} w,
\label{affWeyldefn}
\end{equation}
for $\mu, \nu\in \fa_\ZZ^*$ and $w\in W_0$.

Let $\ell\in \ZZ_{>0}$.  Following \eqref{levelelldotactionfirst}, 
the \emph{level $(-\ell-h)$ dot action of $W$ on $\fa_\ZZ^*$} is given by
\begin{equation}
(t_\mu w)\circ\lambda = (w\circ\lambda) - \ell \mu = w(\lambda+\rho) - \rho - \ell\mu,
\label{levelelldotaction}
\end{equation}
for $\mu\in \fa_\ZZ^*$, $w\in W_0$ and $\lambda\in \fa_\ZZ^*$.

\subsection{The affine Hecke algebra $H$}

Keep the notation for the finite Weyl group $W_0$, the simple reflections $s_1, \ldots, s_n$ 
and the weight lattice $\fa_\ZZ^*$ as in \eqref{wtsWeylgpdefn}.  For $i,j\in \{1, \ldots, n\}$ with $i\ne j$, let
$$m_{ij}\quad\hbox{denote the order of $s_is_j$ in $W_0$}$$
so that $s_i^2=1$ and $(s_is_j)^{m_{ij}}=1$ are the relations for the Coxeter presentation of $W_0$.
The \emph{affine Hecke algebra} is
\begin{equation}
H = \hbox{$\ZZ[t^{\frac12}, t^{-\frac12}]$-span}\{X^\mu T_w\ |\ \mu\in \fa_\ZZ^*, w\in W_0\},
\label{affHeckedefn}
\end{equation}
with $\ZZ[t^{\frac12}, t^{-\frac12}]$ basis $\{X^\mu T_w\ |\ \mu\in \fa_\ZZ^*, w\in W_0\}$ and relations
\begin{equation}
(T_{s_i} - t^{\frac{1}{2}})(T_{s_i} + t^{-\frac{1}{2}}) =0, \qquad
\underbrace{T_{s_i}T_{s_j}T_{s_i}\ldots}_{m_{ij}\ \mathrm{factors}} 
= \underbrace{T_{s_j}T_{s_i}T_{s_j}\ldots}_{m_{ij}\ \mathrm{factors}},
\end{equation}
\begin{equation}
X^{\lambda + \mu} = X^{\lambda}X^{\mu}, \quad\hbox{and}\quad
T_{s_i}X^{\lambda} - X^{s_i\lambda}T_{s_i} = (t^{\frac{1}{2}} - t^{-\frac{1}{2}})\left(\frac{X^{\lambda} - X^{s_i\lambda}}{1 - X^{-\alpha_i}}\right),
\end{equation}
for $i,j\in \{1,\ldots, n\}$ with $i\ne j$ and $\lambda, \mu\in \fa_\ZZ^*$.
The \emph{bar involution on $H$} is the $\ZZ$-linear automorphism 
$\overline{\phantom{T}}\colon H\to H$ given by 
\begin{equation}
\overline{t^{\frac12}} = t^{-\frac12}, \qquad \overline{T_{s_i}} = T_{s_i}^{-1},
\qquad\hbox{and}\qquad
\overline{X^\lambda} = T_{w_0}X^{w_0\lambda}T_{w_0}^{-1}.
\label{baronH}
\end{equation}
for $i=1,\ldots, n$ and
$\lambda, \mu\in  \fa_\ZZ^*$.
For $\mu\in \fa_\ZZ^*$ and $w\in W_0$ define 
\begin{equation}
X^{t_\mu w} = X^\mu (T_{w^{-1}})^{-1}
\qquad\hbox{and}\qquad
T_{t_\mu w} 
=T_x X^{\mu^+}(T_{x^{-1} w})^{-1},
\label{BernsteinCoxeterconversion}
\end{equation}
where $\mu^+$ is the dominant representative of $W_0\mu$ and $x\in W_0$ is minimal length such that
$\mu = x\mu^+$.
\begin{remark}
Formulas \eqref{baronH} and \eqref{BernsteinCoxeterconversion} are just a reformulation of the usual 
bar involution 
and the conversion 
between the Bernstein and Coxeter presentations of the affine Hecke algebra
(see for example \cite[Lemma 2.8 and (1.22)]{NR}).
\end{remark}

\subsection{Definition of $\cP_{-\ell-h}^+$}

Following \eqref{ellAlcovefirst} and \eqref{QG}, define
\begin{equation}
A_{-\ell-h} = \{ \nu\in \fa^*_\ZZ\ |\ 
\hbox{$\langle \nu,\varphi^\vee\rangle \ge -\ell-1$ and
$\langle \nu, \alpha_i^\vee\rangle \le -1$ for $i\in \{1, \ldots, n\}$}\}.
\label{ellAlcove}
\end{equation}
and
\begin{equation}
\cP^+_{-\ell-h} = \bigoplus_{\nu\in A_{-\ell-h}} \varepsilon_0 H\mathbf{1}_\nu,
\label{Heckemoduledefn}
\end{equation}
where $\varepsilon_0$ and $\mathbf{1}_\nu$ are formal symbols satisfying
$\overline{\varepsilon_0}=\varepsilon_0$, 
$\overline{\mathbf{1}_\nu} = \mathbf{1}_\nu$,
$$\hbox{$\varepsilon_0 T_w = (-t^{-\frac12})^{\ell(w)}\varepsilon_0$ for $w\in W_0$}
\qquad\hbox{and}\qquad
\hbox{$T_y\mathbf{1}_\nu = (t^{\frac12})^{\ell(y)}\mathbf{1}_\nu$ for $y\in W_\nu$,}
$$
where $W_\nu = \mathrm{Stab}_W(\nu)$ under the level $(-\ell-h)$ dot action of $W$ on $\fa_\ZZ^*$.
It is important to note  that here the $\mathbf{1}_\nu$ are formal symbols 
(and not elements of the Hecke algebra as in the case of \eqref{projectors})
so that $\mathbf{1}_\nu\ne \mathbf{1}_\gamma$ if 
$\nu\ne \gamma$
(even though it may be that $W_\nu=W_\gamma$).
Define a bar involution 
\begin{equation}
\overline{\phantom{T}}\colon \cP_{-\ell-h}^+\to \cP_{-\ell-h}^+
\qquad\hbox{by}\qquad
\overline{\varepsilon_0 h \mathbf{1}_\nu } = \varepsilon_0 \bar h \mathbf{1}_\nu,
\quad\hbox{for $\nu\in A_{-\ell-h}$ and $h\in H$.}
\label{barforP}
\end{equation}
For $\lambda\in \fa^*_\ZZ$ define
\begin{equation}
[T_\lambda] = [T_{w_0y\circ\nu}] = \varepsilon_0 T_y \mathbf{1}_\nu
\qquad\hbox{and}\qquad
[X_\lambda] = [X_{w_0v\circ\nu}]
= \varepsilon_0 X^v \mathbf{1}_\nu,
\label{bracketTXdefn}
\end{equation}
where 
\begin{equation}
\lambda = w_0y\circ\nu = w_0v\circ\nu,\qquad\hbox{with $\nu\in A_{-\ell-h}$,
\quad and}
\label{Tlambdaalcove}
\end{equation}
\begin{enumerate}
\item[(T)] 
$y\in W$ is such that $T_{yu}=T_{y}T_u$ for any $u\in W_\nu$ 
and 
\item[(X)]
$v\in W$ is 
such that $X^{vu}
=X^{v}T_{u}$ for any $u\in W_{\nu}$.
\end{enumerate}
The condition (T) is equivalent to $y$ being a minimal length representative of the coset $yW_\nu$, i.e.
$y\in W^\nu$.


\subsection{The straightening laws for $[T_\lambda]$}
The following Proposition is a special case of the situation in
Proposition \ref{doublecosetrepbasis}.  As in Proposition \ref{doublecosetrepbasis}, when
$\lambda\in (\fa_\ZZ^*)^+$ ($\lambda$ is a dominant integral weight) then the 
element $[T_\lambda]$ has an 
expansion in $H$ as a sum over the double coset $W_0 u W_\nu$, where $\lambda^+ = w_0u\circ \nu$ with $\nu\in A_{-\ell-h}$
and $u$ is minimal length in $W_0 u W_\nu$.  The properties in Proposition \ref{bracketTlambdaproperties} determine 
$[T_\lambda]$ for $\lambda\in \fa_\ZZ^*$ (all integral weights).

\begin{prop} \label{bracketTlambdaproperties} Let $\lambda\in \fa_\ZZ^*$.
Let $\lambda^+$ be the maximal element of $W_0\circ\lambda$
and let $\lambda^-$ be the minimal element of $W_0\circ\lambda$ in dominance order.
Let $u\in W$ and $x\in W_0$ be of minimal length such that 
$$\lambda^- = u\circ\nu
\qquad\hbox{and}\qquad\lambda = x\circ\lambda^+.$$
Then $[T_\lambda] = (-t^{-\frac12})^{\ell(x)}[T_{\lambda^+}]$ and
$$
[T_{\lambda^+}] = \begin{cases}
\varepsilon_0 T_u\mathbf{1}_\nu, &\hbox{if $\langle \lambda^++\rho, \alpha_i^\vee\rangle \ne 0$ for $i\in \{1,\ldots, n\}$,} \\
0, &\hbox{otherwise.}
\end{cases}
$$
\end{prop}
\begin{proof} 
As in \eqref{bracketTXdefn},
let $y\in W^\nu$ be such that $\lambda = w_0y\circ\nu$.
Then
$$\lambda^- = u\circ\nu \qquad\hbox{and}\qquad
\lambda^+ = w_0u\circ\nu\quad\hbox{and}\quad
y=(w_0xw_0)u,
$$ 
since
$\lambda = x\circ\lambda^+ = xw_0u\circ\nu=w_0(w_0xw_0)u\circ\nu$.
Thus, using the definition in \eqref{bracketTXdefn},
$$
[T_{\lambda^+}] = [T_{w_0u\circ\nu}] = \varepsilon_0 T_u \mathbf{1}_\nu,
\qquad\qquad
[T_{\lambda^-}] = [T_{u\circ\nu}] = [T_{w_0(w_0u)\circ\nu}] = \varepsilon_0 T_{m} \mathbf{1}_\nu, 
$$
where $m$ is the minimal length representative of the coset $w_0uW_\nu$ and
\begin{align*}
[T_\lambda] &= [T_{w_0(w_0xw_0)u\circ\nu}] = \varepsilon_0 T_{w_0xw_0u}\mathbf{1}_\nu
=\varepsilon_0 T_{w_0xw_0}T_u\mathbf{1}_\nu \\
&=(-t^{-\frac12})^{\ell(w_0xw_0)} \varepsilon_0 T_u \mathbf{1}_\nu
=(-t^{-\frac12})^{\ell(x)} \varepsilon_0 T_u \mathbf{1}_\nu.
\end{align*}

If $i\in \{1, \ldots, n\}$ and $\langle \lambda^++\rho, \alpha_i^\vee\rangle=0$ then
$s_j\in W_{\lambda^-}$ where $s_j=w_0s_iw_0$.  Since $W_{\lambda^-}=W_{u\circ\nu}
=uW_\nu u^{-1}$, then $s_{u\alpha_j}=u^{-1}s_ju\in W_\nu$.  Since $\nu\in A_{-\ell-h}$
then $u^{-1}s_ju = s_{u\alpha_j}=s_k$ with $k\in \{0, \ldots, n\}$.  Thus $s_ju=us_k$ and
\begin{align*}
[T_{\lambda^+}] 
&= \varepsilon_0 T_u\mathbf{1}_\nu
= (-t^{\frac12})\varepsilon_0 T_{s_j}T_u\mathbf{1}_\nu
= (-t^{\frac12})\varepsilon_0 T_{s_ju}\mathbf{1}_\nu \\
&= (-t^{\frac12})\varepsilon_0 T_{us_k}\mathbf{1}_\nu
= (-t^{\frac12})\varepsilon_0 T_uT_{s_k}\mathbf{1}_\nu 
= (-t^{\frac12})t^{\frac12}\varepsilon_0 T_u\mathbf{1}_\nu
=-t[T_{\lambda^+}],
\end{align*}
so that $[T_{\lambda^+}]=0$.
\end{proof}

\begin{remark}
The following ``straightening laws'' for $[T_\lambda]$ follow from Proposition 
\ref{bracketTlambdaproperties}.
Let $\lambda\in \fa_\ZZ^*$ and let $i\in \{1, \ldots, n\}$.  Then
\begin{equation}
[T_{s_i\circ\lambda}] = \begin{cases}
-t^{\frac12}[T_{\lambda}], &\hbox{if $\langle \lambda+\rho, \alpha_i^\vee\rangle<0$,} \\
0, &\hbox{if $\langle \lambda+\rho, \alpha_i^\vee\rangle=0$.}
\end{cases}
\label{bracketTlambdastraightening}
\end{equation}
\end{remark}

\subsection{The straightening laws for $[X_\lambda]$}
In parallel with the case for $[T_\lambda]$, 
the properties in Proposition \ref{bracketXlambdaproperties} determine 
$[X_\lambda]$ for $\lambda\in \fa_\ZZ^*$ (all integral weights) in terms of $[X_{\lambda^+}]$
for $\lambda^+\in (\fa_\ZZ^*)^+$ (dominant integral weights).  Proposition \ref{bracketXlambdaproperties} is the same as \cite[Prop.\ 6.3(ii)]{GH}
(see also \cite[Prop. 5.9]{LT}).

\begin{prop} \label{bracketXlambdaproperties} Let $\lambda\in \fa_\ZZ^*$ 
and let $\lambda^+$ and $\lambda^-$ be the dominant and the antidominant representatives
of $W_0\circ \lambda$, respectively.
\begin{enumerate}
\item[(a)] If $i\in \{1, \ldots, n\}$ and $\langle \lambda+\rho, \alpha_i^\vee\rangle=0$ then $[X_\lambda]=0$.
\item[(b)] If $\langle \lambda+\rho, \alpha_i^\vee\rangle\ne 0$ for $i\in \{1, \ldots, n\}$
then $[X_{\lambda^+}] = [T_{\lambda^+}]$.
\item[(c)] Let $i\in \{1, \ldots, n\}$.  Then 
\begin{equation*}
[X_{s_i\circ\lambda}]=\begin{cases}
-[X_\lambda], &\hbox{if $\langle\lambda+\rho,\alpha_i^\vee\rangle \in \ell\ZZ_{\ge 0}$,} \\
-t^{\frac12}[X_\lambda], &\hbox{if $0< \langle\lambda+\rho,\alpha_i^\vee\rangle < \ell$,} \\
-t^{\frac12}[X_{s_i\circ\lambda^{(1)}}] -[X_{\lambda^{(1)}}] - t^{\frac12}[X_\lambda], 
&\hbox{if $ \langle\lambda+\rho,\alpha_i^\vee\rangle > \ell$ and $\langle\lambda+\rho,\alpha_i^\vee\rangle\not\in \ell\ZZ$,}
\end{cases}
\label{XlambdaStraighteningRelations}
\end{equation*}
where 
$$\lambda^{(1)} = \lambda - j\alpha_i
\qquad\hbox{if $\langle \lambda+\rho,\alpha_i^\vee\rangle =k\ell + j$,}
\quad\hbox{with $k\in \ZZ_{\ge 0}$ and 
$j\in \{1, \ldots, \ell-1\}$.}
$$
\end{enumerate}
\end{prop}
\begin{proof}
Define $[X_\lambda] = \varepsilon_0 X^v \mathbf{1}_\nu$ as in \eqref{bracketTXdefn}
and let $\mu\in \fa^*_\ZZ$ and $w\in W_0$ to write
\begin{equation}
v =t_\mu w.
\qquad\hbox{Then}\qquad
[X_\lambda] = \varepsilon_0 X^v\mathbf{1}_\nu
=\varepsilon_0 X^{t_\mu w} \mathbf{1}_\nu
=\varepsilon_0 X^\mu (T_{w^{-1}})^{-1}\mathbf{1}_\nu.
\label{bracketXmuwform}
\end{equation}
The weight $\lambda$ is the $-\ell w_0\mu$-translate of the element $(w_0w)\circ \nu$ since
\begin{equation}
\lambda = w_0v\circ\nu = w_0t_\mu w\circ\nu 
= t_{w_0\mu}(w_0w)\circ\nu = -\ell w_0\mu+(w_0w)\circ\nu.
\label{dottonondot}
\end{equation}
Keeping $i\in \{1, \ldots, n\}$ as in the statement of (c), let
\begin{equation}
s_k = w_0s_iw_0 \quad\hbox{and}\quad \alpha_k = w_0(\alpha_i).
\end{equation}

\smallskip\noindent
(a) follows from the first case of (c): If $\langle \lambda+\rho, \alpha_i^\vee\rangle = 0$ 
then $s_i\circ\lambda=\lambda$ and 
$[X_\lambda]=[X_{s_i\circ\lambda}]=-[X_\lambda]$, so that $2[X_\lambda]=0$.

\smallskip\noindent
(b) Assume $\langle \lambda+\rho,\alpha_i^\vee\rangle \ne 0$ for all $i\in \{1,\ldots, n\}$. 
Let $u\in W$ be of minimal length such that $\lambda^-=u\circ\nu$.
Then 
$\lambda^+=w_0u\circ\nu$ and, by the definition in \eqref{bracketTXdefn},
$[X_{\lambda^+}] = [X_{w_0u\circ\nu}]= \varepsilon_0 X^u \mathbf{1}_\nu$.
Write $u=t_{\mu^+}x$ with $\mu^+\in \fa_\ZZ^*$
and $x\in W_0$.  By \eqref{dottonondot}, $\mu^+$ is dominant since $\lambda^-$ is in the antidominant chamber,
and \eqref{BernsteinCoxeterconversion} then gives that $X^u=T_u$.  Thus
$$[X_{\lambda^+}] 
=[X_{w_0u\circ\nu}]
= \varepsilon_0 X^u \mathbf{1}_\nu
=\varepsilon_0 T_u\mathbf{1}_\nu
=[T_{w_0u\circ\nu}]
=[T_{\lambda^+}].
$$

\smallskip\noindent
(c)  
The proof depends on the following identities in $H$, which we refer to as 
``lifted straightening laws''.
The equality $0=\varepsilon_0 (t^{\frac{1}{2}}+T_{s_k}^{-1})$ is used to establish the ``right half of the hexagon lifted straightening law'':
If $s_kw>w$ then
\begin{align}
0&=\varepsilon_0 (t^{\frac{1}{2}}+T_{s_k}^{-1})(X^{s_k\mu}+ X^{\mu})(T_{w^{-1}})^{-1}
=\varepsilon_0(X^{s_k\mu}+ X^{\mu})(t^{\frac{1}{2}}+T_{s_k}^{-1})T_{w^{-1}}^{-1}  \nonumber \\
&=\varepsilon_0 (X^{s_k\mu}T_{(s_k w)^{-1}}^{-1}
+t^{\frac{1}{2}} X^{s_k\mu}T_{w^{-1}}^{-1}  
+ X^{\mu}T_{(s_k w)^{-1}}^{-1}
+t^{\frac{1}{2}} X^{\mu}T_{w^{-1}}^{-1}).
\label{straighteningA}\tag{R}
\end{align}
The equality
\begin{equation}
0=T_{s_k}X^{s_k\mu}-X^{s_k\mu+\alpha_k}T_{s_k}^{-1}+T_{s_k}X^{\mu-\alpha_k}-X^{\mu}T_{s_k}^{-1},
\label{straighteninglifted}
\end{equation}
is proved by the computation
\begin{align*}
&T_{s_k}X^{s_k\mu}-X^{s_k\mu+\alpha_k}T_{s_k}^{-1}
+ T_{s_k}X^{\mu-\alpha_k} - X^{\mu}T_{s_k}^{-1} \\
&=T_{s_k}X^{s_k\mu}-X^{s_k\mu+\alpha_k}(T_{s_k}
- (t^{\frac12}-t^{-\frac12})) + T_{s_k}X^{\mu-\alpha_k} - X^{\mu}(T_{s_k}-(t^{\frac12}-t^{-\frac12})) \\
&=(t^{\frac12}-t^{-\frac12})\frac{X^{s_k\mu}-X^\mu}{1-X^{-\alpha_k}}
+ X^{s_k\mu+\alpha_k}(t^{\frac12}-t^{-\frac12})
+ (t^{\frac12}-t^{-\frac12})\frac{X^{\mu-\alpha_k}-X^{s_k\mu+\alpha_k}}{1-X^{-\alpha_k}}
+ X^{\mu}(t^{\frac12}-t^{-\frac12}) \\
&= \frac{(t^{\frac12}-t^{-\frac12})}{1-X^{-\alpha_k}}
\left(X^{s_k\mu}-X^\mu +(1-X^{-\alpha_k})X^\mu + X^{\mu-\alpha_k}-X^{s_k\mu+\alpha_k} 
+ (1-X^{-\alpha_k})X^{s_k\mu+\alpha_k}\right) \\
&=0.
\end{align*}
The identity \eqref{straighteninglifted} is 
the source of the ``left half of the hexagon lifted straightening law'':
If $s_kw>w$ then
\begin{align}
0 &=\varepsilon_0 (T_{s_k}X^{s_k\mu}-X^{s_k\mu+\alpha_k}T_{s_k}^{-1}+T_{s_k}X^{\mu-\alpha_k}-X^{\mu}T_{s_k}^{-1}) T_{w^{-1}}^{-1} 
\nonumber \\
&=\varepsilon_0(-t^{-\frac{1}{2}}X^{s_k\mu}T_{w^{-1}}^{-1} -X^{s_k\mu+\alpha_k}T_{(s_kw)^{-1}}^{-1}
- t^{-\frac{1}{2}} X^{\mu-\alpha_k}T_{w^{-1}}^{-1} - X^\mu T_{(s_kw)^{-1}}^{-1}).
\label{straighteningB}\tag{L}
\end{align}


\smallskip\noindent
\emph{Case 1R:}
$0 
 \leq \langle \ell (-w_0\mu), \alpha_i^\vee\rangle -\ell
< \langle \ell (-w_0\mu), \alpha_i^\vee\rangle 
\le \langle \lambda+\rho, \alpha_i^\vee\rangle
< \langle \ell (-w_0\mu), \alpha_i^\vee\rangle +\ell.$

First assume that $\langle \ell(-w_0\mu), \alpha_i^\vee\rangle -\ell
< \langle \ell(-w_0\mu), \alpha_i^\vee\rangle < \langle \lambda+\rho, \alpha_i^\vee\rangle
< \langle \ell (-w_0\mu), \alpha_i^\vee\rangle +\ell$. 
Then (see the upper picture for Case 1R)
$$
\begin{array}{rlrl}
\ [X_{s_i\circ\lambda}]&=\varepsilon_0 X^{s_k\mu}T_{(s_kw)^{-1}}^{-1}\mathbf{1}_\nu, \qquad
&\ [X_\lambda] &= \varepsilon_0 X^\mu T_{w^{-1}}^{-1}\mathbf{1}_\nu,
\\ \\
\ [X_{s_i\circ\lambda^{(1)}}] &= \varepsilon_0 X^{s_k\mu}T_{w^{-1}}^{-1}\mathbf{1}_\nu,
&\ [X_{\lambda^{(1)}}] &= \varepsilon_0 X^\mu T_{(s_kw)^{-1}}^{-1}\mathbf{1}_\nu. 
\end{array}
$$
Since
$$
\langle w\circ\nu + \rho, \alpha_k^\vee\rangle
=\langle w_0w\circ\nu + \rho, \alpha_i^\vee\rangle
=\langle (\lambda-\ell(-w_0\mu))+\rho, \alpha_i^\vee\rangle >0$$ 
then $s_kw>w$ and so equation \eqref{straighteningA} gives 
\begin{align}
0
&=\varepsilon_0 (X^{s_k\mu}T_{(s_k w)^{-1}}^{-1}
+t^{\frac{1}{2}} X^{s_k\mu}T_{w^{-1}}^{-1}  
+ X^{\mu}T_{(s_k w)^{-1}}^{-1}
+t^{\frac{1}{2}} X^{\mu}T_{w^{-1}}^{-1})\mathbf{1}_\nu \nonumber \\
&=[X_{s_i\circ\lambda}]+t^{\frac{1}{2}}[X_{s_i\circ\lambda^{(1)}}]+[X_{\lambda^{(1)}}]+t^{\frac{1}{2}}[X_{\lambda}].
\tag{1Rreg}\label{3areg}
\end{align}

In the limiting case $\langle \ell(-w_0\mu), \alpha_i^\vee\rangle - \ell < \langle \ell (-w_0\mu), \alpha_i^\vee\rangle 
=\langle \lambda+\rho, \alpha_i^\vee\rangle
< \langle \ell (-w_0\mu), \alpha_i^\vee\rangle +\ell$,  
then (see the lower picture for Case 1R)
\begin{equation}
[X_{s_i\circ\lambda}] =\varepsilon_0 X^{s_k \mu} T_{w^{-1}}^{-1}\mathbf{1}_\nu,
\qquad\hbox{and}\qquad
[X_\lambda] = \varepsilon_0 X^\mu T_{w^{-1}}^{-1}\mathbf{1}_\nu
\tag{cen}\label{cen}
\end{equation}
Since
$$
\langle w\circ\nu + \rho, \alpha_k^\vee\rangle
=\langle w_0w\circ\nu + \rho, \alpha_i^\vee\rangle
=\langle (\lambda-\ell(-w_0\mu))+\rho, \alpha_i^\vee\rangle =0$$ 
then $s_k\in W_{w\circ\nu}$ and $w^{-1}s_k w\in W_\nu$.  
Let
$$
s_j=w^{-1}s_kw\in W_\nu
\qquad\hbox{and}\qquad
x=s_kw=ws_j,
$$
so that $s_kx>x$ and $xs_j>x$ and  
$$X^\mu T_{(s_k w)^{-1}}^{-1}T_{s_j}^{-1} = X^\mu T_{w^{-1}}^{-1}
\qquad\hbox{and}\qquad
X^{s_k\mu} T_{(s_kw)^{-1}}^{-1}T_{s_j}^{-1} = X^{s_k\mu} T_{w^{-1}}^{-1}.
$$
Since $s_kx>x$ then equation \eqref{straighteningA} gives 
\begin{align}
0
&=\varepsilon_0 (X^{s_k\mu}T_{(s_k x)^{-1}}^{-1}
+t^{\frac{1}{2}} X^{s_k\mu}T_{x^{-1}}^{-1}  
+ X^{\mu}T_{(s_k x)^{-1}}^{-1}
+t^{\frac{1}{2}} X^{\mu}T_{x^{-1}}^{-1})\mathbf{1}_\nu \nonumber \\
&=\varepsilon_0 (X^{s_k\mu}T_{(s_k x)^{-1}}^{-1}
+t X^{s_k\mu}T_{x^{-1}}^{-1}  T_{s_j}^{-1}
+ X^{\mu}T_{(s_k x)^{-1}}^{-1}
+t X^{\mu}T_{x^{-1}}^{-1}T_{s_j}^{-1})\mathbf{1}_\nu \nonumber \\
&=\varepsilon_0 (X^{s_k\mu}T_{(s_k x)^{-1}}^{-1}
+t X^{s_k\mu}T_{(xs_j)^{-1}}^{-1}  
+ X^{\mu}T_{(s_k x)^{-1}}^{-1}
+t X^{\mu}T_{(x s_j)^{-1}}^{-1})\mathbf{1}_\nu \nonumber \\
&=\varepsilon_0 (X^{s_k\mu}T_{w^{-1}}^{-1}
+t X^{s_k\mu}T_{w^{-1}}^{-1}  
+ X^{\mu}T_{w^{-1}}^{-1}
+t X^{\mu}T_{w^{-1}}^{-1})\mathbf{1}_\nu \nonumber \\
&=(1+t)([X_{s_i\circ\lambda}]+[X_{\lambda}]).
\tag{1Rsing}\label{3asing}
\end{align}


\smallskip\noindent
\emph{Case 1L}: $0\le \langle \ell(-w_0\mu), \alpha_i^\vee\rangle-\ell
\le \langle \lambda+\rho, \alpha_i^\vee\rangle
< \langle \ell (-w_0\mu), \alpha_i^\vee\rangle 
< \langle \ell (-w_0\mu), \alpha_i^\vee\rangle +\ell.$

First assume that $\langle \ell (-w_0\mu), \alpha_i^\vee\rangle-\ell
< \langle \lambda+\rho, \alpha_i^\vee\rangle
< \langle \ell (-w_0\mu), \alpha_i^\vee\rangle< \langle \ell (-w_0\mu), \alpha_i^\vee\rangle+\ell$.
With $x=s_kw$,
(see the upper picture for Case 1L)
$$
\begin{array}{rlrl}
\ [X_{s_i\circ\lambda^{(1)}}] &= \varepsilon_0 X^{s_k\mu+\alpha_k}T_{(s_kx)^{-1}}^{-1}\mathbf{1}_\nu, 
&\ [X_{\lambda^{(1)}}] &= \varepsilon_0 X^{\mu-\alpha_k} T_{x^{-1}}^{-1}\mathbf{1}_\nu,
\\
\ [X_{s_i\circ\lambda}] &=\varepsilon_0 X^{s_k\mu}T_{x^{-1}}^{-1}\mathbf{1}_\nu, \qquad
&\ [X_\lambda] &= \varepsilon_0 X^\mu T_{(s_kx)^{-1}}^{-1}\mathbf{1}_\nu.
\end{array}
$$
Since
$$
\langle w\circ\nu + \rho, \alpha_k^\vee\rangle
=\langle w_0w\circ\nu + \rho, \alpha_i^\vee\rangle
=\langle (\lambda-\ell(-w_0\mu))+\rho, \alpha_i^\vee\rangle <0$$ 
then $s_kw<w$ and so $x<s_kx$.  
Then equation \eqref{straighteningB} gives 
\begin{align}
0 
&=\varepsilon_0(-t^{-\frac{1}{2}}X^{s_k\mu}T_{x^{-1}}^{-1} -X^{s_k\mu+\alpha_k}T_{(s_kx)^{-1}}^{-1}
- t^{-\frac{1}{2}} X^{\mu-\alpha_k}T_{x^{-1}}^{-1} - X^\mu T_{(s_kx)^{-1}}^{-1})\mathbf{1}_\nu 
\nonumber \\
&= - t^{-\frac{1}{2}}([X_{s_i\circ\lambda}] + t^{\frac12}[X_{s_i\circ\lambda^{(1)}}] 
+ [X_{\lambda^{(1)}}] +t^{\frac12} [X_{\lambda}]).
\tag{1Lreg}\label{3breg}
\end{align}

In the limiting case
$\langle \ell(-w_0\gamma), \alpha_i^\vee\rangle - \ell = \langle \lambda+\rho, \alpha_i^\vee\rangle
< \langle \ell(-w_0\gamma), \alpha_i^\vee\rangle < \langle \ell(-w_0\gamma), \alpha_i^\vee\rangle + \ell$
with
\begin{equation}
\gamma = \mu+\alpha_k,
\tag{bdy}\label{bdy}
\end{equation}
then (see the lower picture for Case 1L)
\begin{align*}
[X_{s_i\circ \lambda}] 
&= \varepsilon_0 X^{s_k\mu-\alpha_k}T_{w^{-1}}^{-1} \mathbf{1}_\nu 
= \varepsilon_0 X^{s_k\gamma}T_{w^{-1}}^{-1} \mathbf{1}_\nu 
\qquad\hbox{and} \\
[X_\lambda] 
&= \varepsilon_0 X^{\mu}T_{w^{-1}}^{-1} \mathbf{1}_\nu
= \varepsilon_0 X^{\gamma-\alpha_k}T_{w^{-1}}^{-1} \mathbf{1}_\nu.
\end{align*}
Since
$$
\langle w\circ\nu + \rho, \alpha_k^\vee\rangle
=\langle w_0w\circ\nu + \rho, \alpha_i^\vee\rangle
=\langle (\lambda-\ell(-w_0\mu))+\rho, \alpha_i^\vee\rangle >0$$ 
then $s_kw>w$.
Since
$$
\langle w\circ\nu + \rho, \alpha_k^\vee\rangle-\ell
=\langle w_0w\circ\nu + \rho, \alpha_i^\vee\rangle-\ell
=\langle (\lambda-\ell(-w_0\mu))+\rho, \alpha_i^\vee\rangle-\ell =0$$ 
then $s_{-\alpha_k+\delta}\in W_{w\circ\nu}$ and 
$s_0=ws_{-\alpha_k+\delta}w^{-1}=s_{w(-\alpha_k+\delta)}\in W_\nu$.
Then
$$X^{s_k\gamma+\alpha_k}T_{(s_kw)^{-1}}^{-1} = X^{s_k\gamma}T_{w^{-1}}^{-1}T_{s_0}
\qquad\hbox{and}\qquad
X^{\gamma} T_{(s_kw)^{-1}}^{-1} = X^{\gamma-\alpha_k} T_{w^{-1}}^{-1}T_{s_0}.
$$
and equation \eqref{straighteningB} gives
\begin{align}
0 
&=\varepsilon_0(-t^{-\frac{1}{2}}X^{s_k\gamma}T_{w^{-1}}^{-1} -X^{s_k\gamma+\alpha_k}T_{(s_kw)^{-1}}^{-1}
- t^{-\frac{1}{2}} X^{\gamma-\alpha_k}T_{w^{-1}}^{-1} - X^\gamma T_{(s_kw)^{-1}}^{-1})\mathbf{1}_\nu \nonumber \\
&=\varepsilon_0(-t^{-\frac{1}{2}}X^{s_k\gamma}T_{w^{-1}}^{-1} -X^{s_k\gamma}T_{w^{-1}}^{-1}T_{s_0}
- t^{-\frac{1}{2}} X^{\gamma-\alpha_k}T_{w^{-1}}^{-1} - X^{\gamma-\alpha_k} T_{w^{-1}}^{-1}T_{s_0})\mathbf{1}_\nu \nonumber \\
&=\varepsilon_0(-t^{-\frac{1}{2}}X^{s_k\gamma}T_{w^{-1}}^{-1} -t^{\frac12}X^{s_k\gamma}T_{w^{-1}}^{-1}
- t^{-\frac{1}{2}} X^{\gamma-\alpha_k}T_{w^{-1}}^{-1} - t^{\frac12}X^{\gamma-\alpha_k} T_{w^{-1}}^{-1})\mathbf{1}_\nu \nonumber \\
&= -(t^{-\frac{1}{2}}+t^{\frac12})([X_{s_i\circ\lambda}] + [X_{\lambda}]).
\tag{1Lsing}\label{3bsing}
\end{align}

\smallskip\noindent
\emph{Case 2R}: $0=\langle \ell(-w_0\mu), \alpha_i^\vee\rangle$ and 
$0=\langle \ell(-w_0\mu), \alpha_i^\vee\rangle \le \langle \lambda+\rho, \alpha_i^\vee\rangle 
< \ell= \langle \ell(-w_0\mu), \alpha_i^\vee\rangle+\ell.$
This case is really a special case of Case 1R, with
$$s_k\mu = \mu,
\qquad \hbox{since}\quad
0=\langle \ell(-w_0\mu), \alpha_i^\vee\rangle=\ell\langle -\mu, \alpha_k^\vee\rangle.
$$
In the case that $0< \langle \lambda+\rho, \alpha_i^\vee\rangle<\ell$
then (see the top picture in Case 2R)
$$
[X_{s_i\circ\lambda}] 
= \varepsilon_0 X^{\mu} T_{(s_kw)^{-1}}^{-1}\mathbf{1}_\nu
\qquad\hbox{and}\qquad
[X_\lambda] = \varepsilon_0 X^\mu T_{w^{-1}}\mathbf{1}_\nu
$$
and \eqref{3areg} becomes
\begin{align}
0
&=\varepsilon_0 (X^{s_k\mu}T_{(s_k w)^{-1}}^{-1}
+t^{\frac{1}{2}} X^{s_k\mu}T_{w^{-1}}^{-1}  
+ X^{\mu}T_{(s_k w)^{-1}}^{-1}
+t^{\frac{1}{2}} X^{\mu}T_{w^{-1}}^{-1})\mathbf{1}_\nu \nonumber \\
&=[X_{s_i\circ\lambda}]+t^{\frac12}[X_\lambda]+[X_{s_i\circ\lambda}]+t^{\frac12}[X_\lambda] 
= 2(t^{\frac{1}{2}}[X_{\lambda}]+[X_{s_i\circ\lambda}]).
\tag{2Rreg}\label{2areg}
\end{align}
For the limiting case where $0=\langle \lambda+\rho, \alpha_i^\vee\rangle <\ell$ (this is analogous to \eqref{cen})
$$[X_\lambda]=[X_{s_i\circ\lambda}]=\varepsilon_0 X^\mu T^{-1}_{w^{-1}} \mathbf{1}_\nu.$$
and \eqref{3asing} becomes
\begin{align}
0
&=\varepsilon_0 (X^{s_k\mu}T_{w^{-1}}^{-1}
+t X^{s_k\mu}T_{w^{-1}}^{-1}  
+ X^{\mu}T_{w^{-1}}^{-1}
+t X^{\mu}T_{w^{-1}}^{-1})\mathbf{1}_\nu  
=(1+t)2[X_{\lambda}].
\tag{2Rsing}\label{2asing}
\end{align}

\smallskip\noindent
\emph{Case 2L}: $\ell=\langle \ell(-w_0\mu), \alpha_i^\vee\rangle$ and
$0=\langle \ell(-w_0\mu), \alpha_i^\vee\rangle - \ell
<\langle \lambda+\rho, \alpha_i\rangle < \langle \ell(-w_0\mu), \alpha_i^\vee\rangle=\ell$.
This case is really a special case of Case 1L, with
$$s_k\mu = \mu-\alpha_k,
\qquad \hbox{since}\quad
1=\hbox{$\frac{1}{\ell}$}\ell 
= \hbox{$\frac{1}{\ell}$}\langle \ell(-w_0\mu), \alpha_i^\vee\rangle
=\langle -\mu, \alpha_k^\vee\rangle.
$$
In the case that $0< \langle \lambda+\rho, \alpha_i^\vee\rangle<\ell$
then 
(see the bottom picture in Case 2L)
$$
[X_{s_i\circ\lambda}] = \varepsilon_0 X^{\mu-\alpha_k}T_{x^{-1}}^{-1}\mathbf{1}_\nu
\qquad\hbox{and}\qquad
[X_\lambda] = \varepsilon_0 X^\mu T_{(s_kx)^{-1}}^{-1} \mathbf{1}_\nu.
$$
and \eqref{3breg} becomes
\begin{align}
0 
&=\varepsilon_0(-t^{-\frac{1}{2}}X^{s_k\mu}T_{x^{-1}}^{-1} 
-X^{s_k\mu+\alpha_k}T_{(s_kx)^{-1}}^{-1}
- t^{-\frac{1}{2}} X^{\mu-\alpha_k}T_{x^{-1}}^{-1} - X^\mu T_{(s_kx)^{-1}}^{-1})\mathbf{1}_\nu 
\nonumber \\
&=-t^{-\frac{1}{2}}[X_{s_i\circ\lambda}] -[X_\lambda] - t^{-\frac{1}{2}} [X_{s_i\circ\lambda}] - [X_\lambda] 
= -2t^{-\frac{1}{2}}([X_{s_i\circ\lambda}]+t^{\frac12}[X_{\lambda}]).
\tag{2Lreg}\label{2breg}
\end{align}
For the limiting case where $0=\langle \lambda+\rho, \alpha_i^\vee\rangle<\ell$ (this is analogous to \eqref{bdy})
$$
[X_\lambda] = [X_{s_i\circ\lambda}] = \varepsilon_0 X^\mu T_{(s_kx)^{-1}}^{-1} \mathbf{1}_\nu
=\varepsilon_0X^{\gamma-\alpha_k}T_{w^{-1}}^{-1}\mathbf{1}_\nu
=\varepsilon_0X^{s_k\gamma}T_{w^{-1}}^{-1}
$$
and \eqref{3bsing} becomes
\begin{align}
0 
&=\varepsilon_0(-t^{-\frac{1}{2}}X^{s_k\gamma}T_{w^{-1}}^{-1} -t^{\frac12}X^{s_k\gamma}T_{w^{-1}}^{-1}
- t^{-\frac{1}{2}} X^{\gamma-\alpha_k}T_{w^{-1}}^{-1} - t^{\frac12}X^{\gamma-\alpha_k} T_{w^{-1}}^{-1})\mathbf{1}_\nu \nonumber \\
&= -(t^{-\frac{1}{2}}+t^{\frac12})2[X_{\lambda}].
\tag{2Lsing}\label{2bsing}
\end{align}

Together these computations complete the proof of part (c):  the third case follows from
\eqref{3areg} and \eqref{3breg}, the second case from \eqref{2areg} and
\eqref{2breg}, and the first case from \eqref{3asing} and \eqref{3bsing},
with \eqref{2asing} and \eqref{2bsing} specifically treating the statement in (a).
\end{proof}

\begin{remark}  If $\lambda \in \ell\fa^*_\ZZ - \rho$ then there is a unique $\mu\in \fa^*_\ZZ$ such that
$$\lambda = \ell w_0\mu - \rho = t_{w_0\mu}\circ (-\rho)=t_{w_0\mu}w_0\circ(-\rho)
=w_0t_{\mu}\circ (-\rho),$$
so that $\lambda=w_0v\circ \nu$ with $\nu=-\rho$ and $v=t_\mu$.
Since $\nu = -\rho$ then $\mathbf{1}_\nu = \mathbf{1}_0$ with
$T_{s_i}\mathbf{1}_0=t^{\frac12}\mathbf{1}_0$ for $i\in \{1, \ldots, n\}$.
Thus,
$$[T_\lambda] = \varepsilon_0 T_{t_\mu} \mathbf{1}_0
\qquad\hbox{and}\qquad
[X_\lambda] = \varepsilon_0 X^\mu \mathbf{1}_0.$$
so that the $[X_\lambda]$, for $\lambda \in \ell\fa^*_\ZZ-\rho$, are the 
elements $A_\mu$ studied in \cite[\S 2]{NR}.  In this case
the first case of Proposition \ref{bracketXlambdaproperties}(c) is the 
straightening law and this 
coincides with the equality  
$A_{s_i\mu}=-A_{\mu}$ proved in \cite[Prop. 2.1]{NR}. 
\end{remark}

\begin{remark}  Following the definition of $[X_\lambda]$ in \eqref{bracketTXdefn},
$$\hbox{if}\quad \lambda = w_0v\circ\nu
\qquad\hbox{then}\quad
w_0\circ\lambda = w_0(w_0v)\circ\nu = w_0(w_0v w_\nu)\circ\nu$$
and we have $[X_\lambda] = \varepsilon_0X^v \mathbf{1}_\nu$ and $[X_{w_0\circ\lambda}]
=\varepsilon_0 X^{w_0vw_\nu} \mathbf{1}_\nu$.  With $X^v = X^{t_\mu w}$ then
\begin{align*}
\overline{X^v} = \overline{X^{t_{\mu}w}}&=\overline{X^\mu (T_{w^{-1}}^{-1})}=\overline{X^{\mu}}T_w=T_{w_0} X^{w_0\mu}T_{w_0}^{-1}T_w=T_{w_0}X^{w_0\mu} T_{w^{-1}w_0}^{-1}\\
&=T_{w_0}X^{t_{w_0\mu}(w_0w)}=T_{w_0}X^{w_0v}=T_{w_0}X^{w_0v w_\nu}T_{w_\nu}
\end{align*}
By the previous computation,  $X^{w_0vw_\nu}= T_{w_0}^{-1} \overline{X^v} T_{w_\nu}^{-1}$, so that
\begin{align*}
[X_{w_0\circ\lambda}]
&=\varepsilon_0 X^{w_0vw_\nu} \mathbf{1}_\nu
=\varepsilon_0 T_{w_0}^{-1} \overline{X^v} T_{w_\nu}^{-1} \mathbf{1}_\nu 
=(-t^{-\frac12})^{-\ell(w_0)}(t^{\frac12})^{-\ell(w_\nu)}
\varepsilon_0 \overline{X^v} \mathbf{1}_\nu \\
&=(-1)^{\ell(w_0)}(t^{-\frac12})^{-\ell(w_0)+\ell(w_\nu)}
\overline{\varepsilon_0 X^v \mathbf{1}_\nu}
=(-1)^{\ell(w_0)}(t^{-\frac12})^{-\ell(w_0)+\ell(w_\nu)}
\overline{[X_\lambda]}.
\end{align*}
Hence 
\begin{equation}
\overline{[X_\lambda]}
=(-1)^{\ell(w_0)}(t^{-\frac12})^{\ell(w_0)-\ell(w_\nu)}[X_{w_0\circ\lambda}].
\label{barbracketXlambda}
\end{equation}
\end{remark}

\subsection{Relating the KL-modules $\cP^+_{-\ell-h}$ and $\cF_\ell$}

In this subsection we tie together our components: the module with bar involution $\cP^+_{-\ell-h}$ from \eqref{Heckemoduledefn} which was built from the affine Hecke algebra and the abstract Fock space $\cF_\ell$
from \eqref{Fstraightening}.  Because of the way that we arrived at $\cP^+_{-\ell-h}$ from representation theory
(see (QG) at the end of Section 3) the isomorphism between $\cP^+_{-\ell-h}$ and $\cF_\ell$
will allow us to prove that the abstract Fock space $\cF_\ell$ captures decomposition numbers of Weyl 
modules for quantum groups at roots of unity.

\begin{thm} \label{abstracttoHecke} Let $\le$ be the dominance order on the set $(\fa_\ZZ^*)^+$ of
dominant integral weights.
$$\begin{array}{lll}
\hbox{Let \quad $\cP_{-\ell-h}^+$} \quad &\hbox{with basis \quad $B=\{ [X_\lambda]\ |\ \lambda\in (\fa_\ZZ^*)^+\}$}
\quad &\hbox{and bar involution 
as in \eqref{barforP}, and} \\
\hbox{let \quad $\cF_\ell$} \quad &\hbox{with basis \quad $\cL = \{ \vert\lambda\rangle\ |\ \lambda\in (\fa_\ZZ^*)^+\}$}
\quad &\hbox{and bar involution 
as in \eqref{Fellbar}.}
\end{array}
$$  
Then $\cP_{-\ell-h}^+$ is a KL-module and 
$$\begin{matrix}
\cP^+_{-\ell-h} &\longrightarrow &\cF_\ell \\
[X_\lambda] &\longmapsto &\vert \lambda\rangle
\end{matrix}
\qquad\hbox{is a KL-module isomorphism.}
$$
\end{thm}
\begin{proof}
By definition (see \eqref{Heckemoduledefn}), 
$\cP^+_{-\ell-h} = \bigoplus_{\nu\in A_{-\ell-h}} \varepsilon_0 H\mathbf{1}_\nu$.  By \S \ref{singparabolicKLsection},
each summand is a KL-module and so $\cP_{-\ell-h}^+$ is a KL-module.

The $\ZZ[t^{\frac12}, t^{-\frac12}]$-module $\cF_\ell$ is generated by $\vert\lambda\rangle$, $\lambda\in \fa_\ZZ^*$.
By definition, these symbols satisfy the relations in \eqref{Fstraightening}.  The $\ZZ[t^{\frac12}, t^{-\frac12}]$-module
$\cP_{-\ell-h}^+$ is generated by the symbols $[X_\lambda]$, $\lambda\in \fa_\ZZ^*$.  By comparison of
the relations in \eqref{Fstraightening} with those in Proposition \ref{bracketXlambdaproperties}(c), there is a surjective $\ZZ[t^{\frac12}, t^{-\frac12}]$-module
homomorphism
\begin{equation}\Phi\colon \cF_\ell \to \cP_{-\ell-h}^+
\quad\hbox{given by}\quad \Phi(\vert\lambda\rangle) = [X_\lambda],
\label{isocFtocP}
\end{equation}
for $\lambda\in \fa_\ZZ^*$.  This homomorphism respects the bar involution since, by 
\eqref{Fellbar}
and
\eqref{barbracketXlambda},
\begin{align*}
\Phi(\overline{\vert\lambda\rangle})
&= \Phi((-1)^{\ell(w_0)}(t^{-\frac12})^{\ell(w_{0})-N_\lambda}\, \vert w_0\circ \lambda \rangle) \\
&= (-1)^{\ell(w_0)}(t^{-\frac12})^{\ell(w_{0})-N_\lambda}\, \Phi(\vert w_0\circ \lambda \rangle) \\
&= (-1)^{\ell(w_0)}(t^{-\frac12})^{\ell(w_{0})-N_\lambda} [X_{w_0\circ\lambda}] \\
&= (-1)^{\ell(w_0)}(t^{-\frac12})^{\ell(w_{0})-N_\lambda}(-1)^{\ell(w_0)}(t^{-\frac12})^{-\ell(w_0)+\ell(w_\nu)}
\overline{[X_\lambda]} \\
&= (-1)^{\ell(w_0)}(t^{-\frac12})^{\ell(w_{0})-N_\lambda}(-1)^{\ell(w_0)}(t^{-\frac12})^{-\ell(w_0)+\ell(w_\nu)}
\overline{\Phi(\vert\lambda\rangle)} = \overline{\Phi(\vert\lambda\rangle)}.
\end{align*}

If $\lambda\in (\fa_\ZZ^*)^+$ then $[X_\lambda] = [T_\lambda]$.  Thus, by 
Proposition \ref{bracketTlambdaproperties} (see also Proposition \ref{doublecosetrepbasis}), the set
$\{ [X_\lambda]\ |\ \lambda\in (\fa_\ZZ^*)^+\}$ is a basis of $\cP_{-\ell-h}$.  Since the $\Phi_\ell$ image of
$\{ \vert \lambda\rangle\ |\ \lambda\in (\fa_\ZZ^*)^+\}$ is linearly independent in $\cP_{-\ell-h}$ this set 
must be linearly independent in $\cF_\ell$ and $\Phi_\ell$ is injective.
Since $\cF_\ell$ is spanned by $\{\vert \lambda\rangle\ |\ \lambda\in (\fa_\ZZ^*)^+\}$ then 
$\Phi_\ell$ is a KL-module isomorphism.
\end{proof}


The KL-module $\cF_\ell$ has 
\begin{equation}
\hbox{standard basis}\quad
\{ \vert\lambda\rangle\ |\ \lambda\in (\fa_\ZZ^*)^+\}
\qquad\hbox{and}\qquad
\hbox{KL-basis}\quad \{ C_\lambda\ |\ \lambda\in (\fa_\ZZ^*)^+\}.
\end{equation}
For $\mu, \lambda\in (\fa_\ZZ^*)^+$ define $p_{\mu\lambda}, d_{\lambda\mu}\in \ZZ[t^{\frac12}, t^{-\frac12}]$ by
\begin{equation}
C_\mu = \vert \mu\rangle + \sum_\mu p_{\mu\lambda} \vert\lambda\rangle,
\qquad\hbox{and}\qquad
\vert \lambda\rangle  = C_\lambda + \sum_{\mu} d_{\lambda\mu}C_\mu.
\label{Felltransitionmatrices}
\end{equation}
The following theorem relates the $p_{\mu\lambda}$ to affine KL-polynomials and the $d_{\lambda\mu}$ to
decomposition numbers of Weyl modules for the quantum group at an $\ell$th root of unity.  It is a
generalization of a type $GL_n$ statement found, for example, in \cite[Thm.\ 6.4]{Sh}.

\begin{thm} Fix $\ell\in \ZZ_{>0}$ and let $q^2$ be a primitive $\ell$th root of unity in $\CC$.  
Let $U_q(\fg)$ be the
Drinfeld-Jimbo quantum group corresponding to the weight lattice $\fa_\ZZ^*$, the Weyl group $W_0$ and the positive roots $R^+$.
Let $L_q(\mu)$ be the simple module of highest weight $\mu$ for the quantum group $U_q(\fg)$ and
let
$$\Delta_q(\lambda)=\Delta_q(\lambda)^{(0)} 
\supseteq \Delta_q(\lambda)^{(1)}\supseteq
\cdots
\qquad\hbox{be the Jantzen filtration}
$$
of the Weyl module $\Delta_q(\lambda)$ of highest weight $\lambda$ for $U_q(\fg)$.

Let $W$ be the affine Weyl group and let $\lambda,\mu\in (\fa_\ZZ^*)^+$ and let $p_{\mu\lambda}$ and $d_{\lambda\mu}$ be as given in \eqref{Felltransitionmatrices}.
\item[(a)] If $\lambda$ and $\mu$ are not in the same $W$-orbit for
the level $(-\ell-h)$ dot-action of $W$ on $\fa_\ZZ^*$ then $d_{\lambda\mu}=0$ and $p_{\mu\lambda} = 0$.  
\item[(b)] If $\lambda$ and $\mu$ are in the same $W$-orbit then let $\nu\in A_{-\ell-h}$ and $x,y\in W$ be such that
$$\lambda = w_0x\circ \nu, \qquad \mu = w_0y\circ\nu,\qquad
x,y\in {}^0W\quad\hbox{and}\quad w_0x,w_0y\in W^\nu,
$$
where $w_0$ is the longest element of the Weyl group $W_0$,
$W_\nu$ is the stabilizer of $\nu$ under the dot action of $W$,
$W^\nu$ is the set of minimal length representatives of cosets in $W/W_\nu$
and
${}^0W$ is the set of minimal length representatives of cosets in $W_0\backslash W$.

\noindent
Then 
\begin{align*}
p_{\mu\lambda} 
& (-1)^{\ell(w_0 x)-\ell(w_0 y)} P^\nu_{w_0 y,w_0 x}(t^{\frac12})
\qquad \hbox{and}  \\
d_{\lambda\mu}
&=\left(\sum_{j\in \ZZ_{\ge 0}} 
t^j \dim\Big(\Hom\Big(\frac{\Delta_q(\lambda)^{(j)}}{ \Delta_q(\lambda)^{(j+1)}},
L_q(\mu)\Big)\Big) \right),
\end{align*}
 where $P^\nu_{w_0 y,w_0 x}(t^{\frac12})$ is the (parabolic singular)
Kazhdan-Lusztig polynomial
 (see \eqref{LusztigSP}) for the affine Hecke algebra $H$ corresponding to $W$ 
(see \eqref{LusztigC} and \eqref{affHeckedefn}).
\end{thm}
\begin{proof} By definition (see \eqref{Heckemoduledefn}), 
$\cP^+_{-\ell-h} = \bigoplus_{\nu\in A_{-\ell-h}} \varepsilon_0 H\mathbf{1}_\nu$.  The analysis in 
\S \ref{singparabolicKLsection} applies to each of the summands $\varepsilon_0 H\mathbf{1}_\nu$ to give that, for $\lambda,\mu\in (\fa_\ZZ^*)^+$,
\begin{align*}
\Phi(\vert\lambda\rangle) = [X_\lambda] = [T_\lambda] = \varepsilon_0 T_y \mathbf{1}_\nu
\qquad\hbox{and}\qquad
\Phi(C_\mu) = C_{w_0x}\mathbf{1}_\nu,
\end{align*}
where $\Phi\colon \cF_\ell \to \cP_{-\ell-h}^+$ is the KL-module isomorphism from \eqref{isocFtocP} and
$x,y\in W$ and $\nu\in A_{-\ell-h}$ are as defined in the statement of (b).
In particular, by \eqref{LusztigSP}, 
the transition matrix between these bases is given by
\begin{align*}
\Phi(C_\mu) 
& = C_{w_0 x}\mathbf{1}_\nu 
=\sum_{w_0 y\le w_0 x\atop y\in {}^0 W, w_0y\in W^\nu}
(-1)^{\ell(w_0 x)-\ell(w_0 y)} P^\nu_{w_0 y,w_0 x}(t^{\frac12}) \varepsilon_0 T_y\mathbf{1}_\nu  \\
& =\sum_{w_0 y\le w_0 x\atop y\in {}^0 W, y\in W^\nu}
(-1)^{\ell(w_0 x)-\ell(w_0 y)} P^\nu_{w_0 y,w_0 x}(t^{\frac12}) \Phi(\vert \lambda\rangle), 
\end{align*}
and, since $\Phi$ is an isomorphism, this establishes the formula for $p_{\mu\lambda}$.
The formula for $d_{\lambda\mu}$ is then a consequence of the isomorphism of (QG) given at the
end of Section 3.
\end{proof}


\newpage

\begin{align*}
\includegraphics[height=10cm]{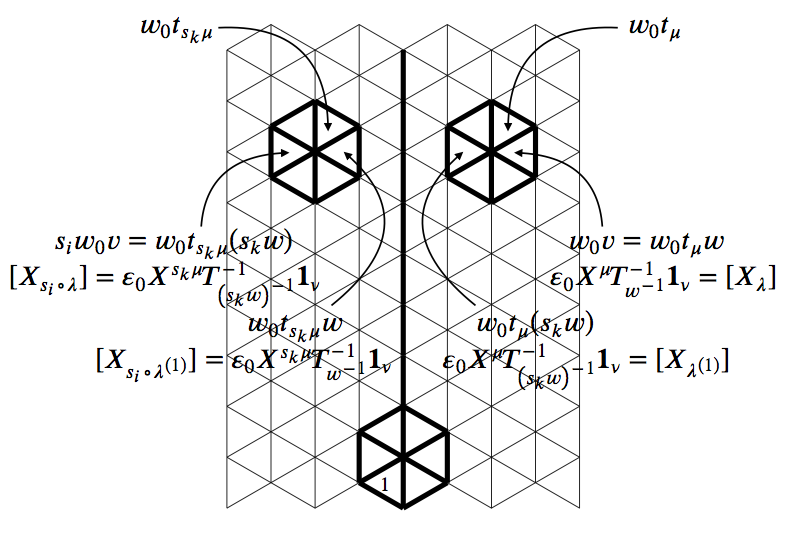} \\
\includegraphics[height=10cm]{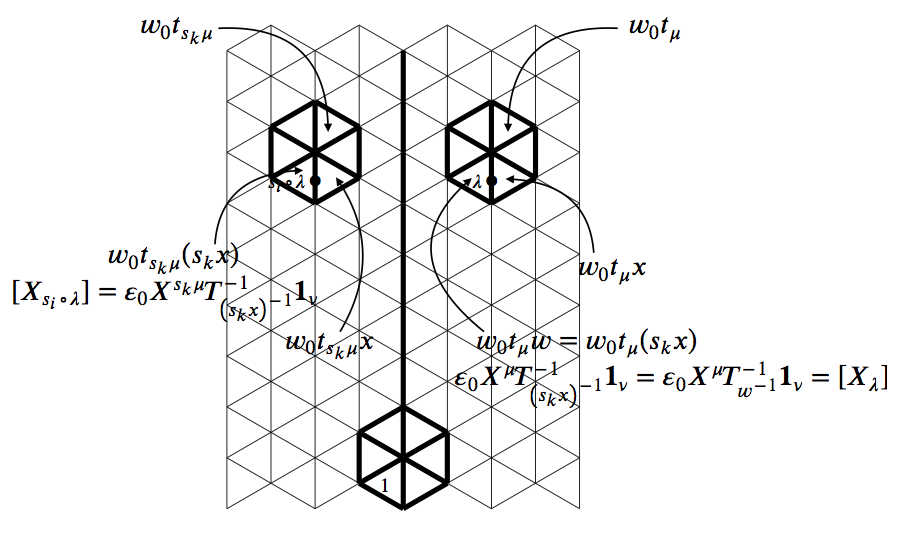} \\
\hbox{Case 1R:
$0 \le \langle \ell(-w_0\mu), \alpha_i^\vee\rangle - \ell 
< \langle \ell(-w_0\mu), \alpha_i^\vee\rangle 
< \langle \lambda+\rho, \alpha_i^\vee\rangle
< \langle \ell (-w_0\mu), \alpha_i^\vee\rangle +\ell$ and}\\
\hbox{\qquad $0 
< \langle \ell (-w_0\mu), \alpha_i^\vee\rangle -\ell
< \langle \ell (-w_0\mu), \alpha_i^\vee\rangle 
= \langle \lambda+\rho, \alpha_i^\vee\rangle
< \langle \ell (-w_0\mu), \alpha_i^\vee\rangle +\ell.$}
\qquad\qquad
\end{align*}

\newpage

\begin{align*}
\includegraphics[height=10cm]{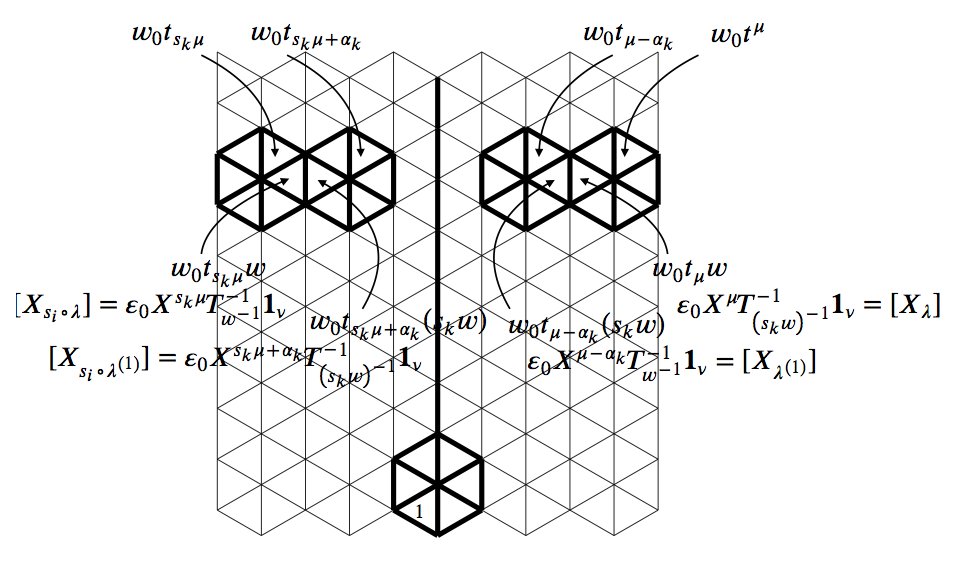} \\
\includegraphics[height=10cm]{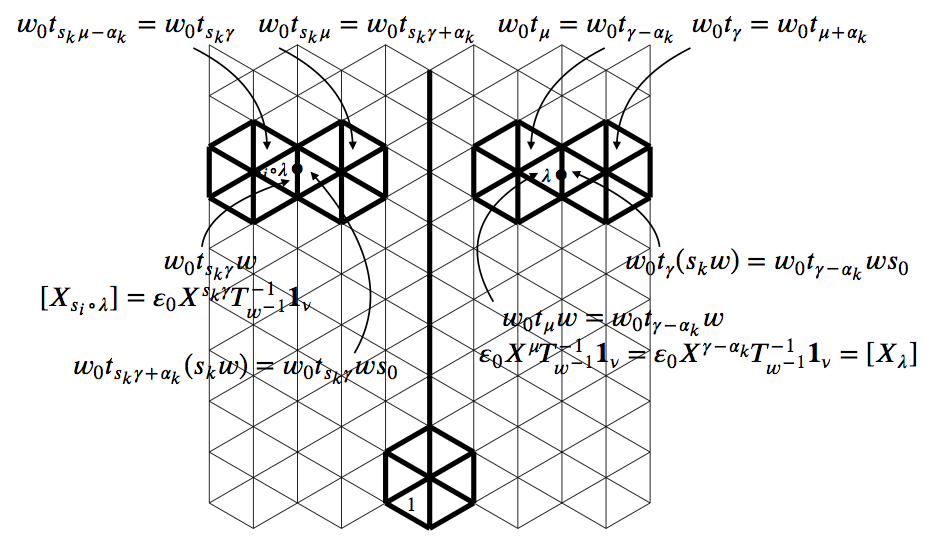} \\
\hbox{Case 1L: 
$0<\langle \ell(-w_0\mu), \alpha_i^\vee\rangle -\ell
< \langle \lambda+\rho, \alpha_i^\vee\rangle
< \langle \ell(-w_0\mu), \alpha_i^\vee\rangle
< \langle \ell(-w_0\mu), \alpha_i^\vee\rangle+\ell
$ and} \\
\hbox{
$0<\langle \ell(-w_0\mu), \alpha_i^\vee\rangle - \ell 
= \langle \lambda+\rho, \alpha_i^\vee\rangle
< \langle \ell(-w_0\mu), \alpha_i^\vee\rangle
<\langle \ell(-w_0\mu), \alpha_i^\vee\rangle+\ell.$} 
\end{align*}

\newpage

\begin{align*}
\includegraphics[height=10cm]{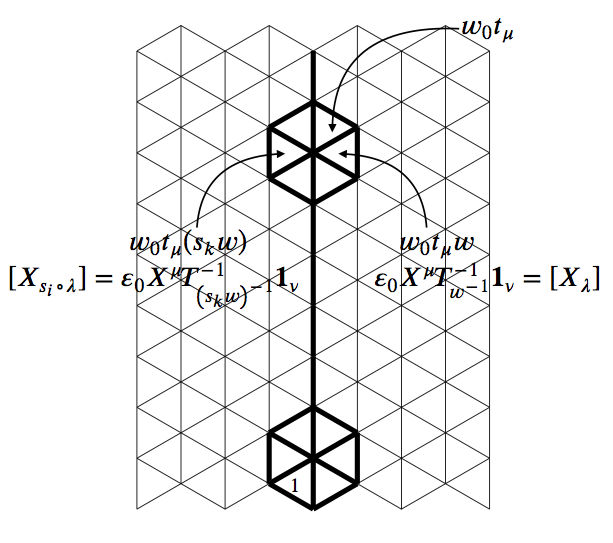} \\
\includegraphics[height=10cm]{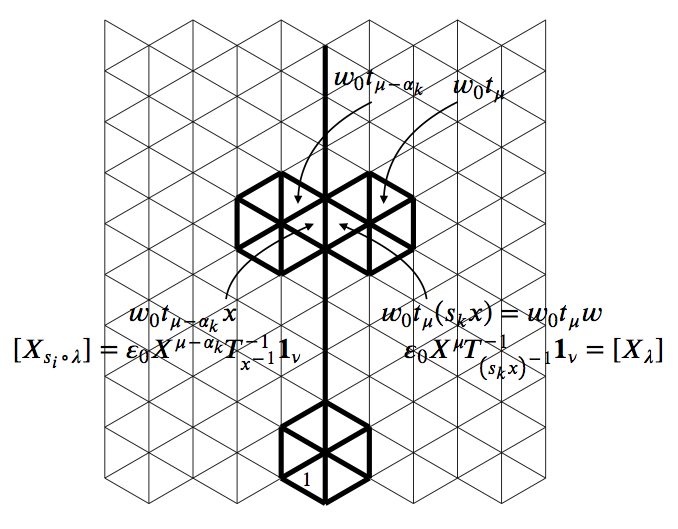} \\
\hbox{Case 2R: 
$
0 = \langle \ell(-w_0\mu)+\rho, \alpha_i^\vee\rangle 
< \langle \lambda+\rho, \alpha_i^\vee\rangle
< \ell = \langle \ell(-w_0\mu)+\rho, \alpha_i^\vee\rangle +\ell$ and} \\
\hbox{Case 2L: 
$0 = \langle \ell(-w_0\mu)+\rho, \alpha_i^\vee\rangle -\ell
< \langle \lambda+\rho, \alpha_i^\vee\rangle
< \ell = \langle \ell(w_0\mu)+\rho, \alpha_i^\vee\rangle$}\phantom{\hbox{and}}
\end{align*}

\end{document}